\documentclass[a4paper,10pt]{amsart}
\usepackage{amsmath,amssymb,hyperref,stmaryrd}

\theoremstyle{plain}
\newtheorem{theorem}{\bf Theorem}[section]
\newtheorem{proposition}[theorem]{\bf Proposition}
\newtheorem{lemma}[theorem]{\bf Lemma}
\newtheorem{corollary}[theorem]{\bf Corollary}
\theoremstyle{definition}
\newtheorem{example}[theorem]{\bf Example}
\newtheorem{remark}[theorem]{\bf Remark}

\begin{document}
\title{On transfer Krull monoids}
\address{Institut f\"ur Mathematik und Wissenschaftliches Rechnen, Karl-Franzens-Universit\"at Graz, NAWI Graz, Heinrichstra{\ss}e 36, 8010 Graz, Austria}
\email{aqsa.bashir@uni-graz.at,andreas.reinhart@uni-graz.at}
\author{Aqsa Bashir and Andreas Reinhart}
\thanks{This work was supported by the Austrian Science Fund FWF, Project Numbers P33499 and W1230.}

\begin{abstract}
Let $H$ be a cancellative commutative monoid, let $\mathcal{A}(H)$ be the set of atoms of $H$ and let $\widetilde{H}$ be the root closure of $H$. Then $H$ is called transfer Krull if there exists a transfer homomorphism from $H$ into a Krull monoid. It is well known that both half-factorial monoids and Krull monoids are transfer Krull monoids. In spite of many examples and counter examples of transfer Krull monoids (that are neither Krull nor half-factorial), transfer Krull monoids have not been studied systematically (so far) as objects on their own. The main goal of the present paper is to attempt the first in-depth study of transfer Krull monoids. We investigate how the root closure of a monoid can affect the transfer Krull property and under what circumstances transfer Krull monoids have to be half-factorial or Krull. In particular, we show that if $\widetilde{H}$ is a DVM, then $H$ is transfer Krull if and only if $H\subseteq\widetilde{H}$ is inert. Moreover, we prove that if $\widetilde{H}$ is factorial, then $H$ is transfer Krull if and only if $\mathcal{A}(\widetilde{H})=\{u\varepsilon\mid u\in\mathcal{A}(H),\varepsilon\in\widetilde{H}^{\times}\}$. We also show that if $\widetilde{H}$ is half-factorial, then $H$ is transfer Krull if and only if $\mathcal{A}(H)\subseteq\mathcal{A}(\widetilde{H})$. Finally, we point out that characterizing the transfer Krull property is more intricate for monoids whose root closure is Krull. This is done by providing a series of counterexamples involving reduced affine monoids.
\end{abstract}

\maketitle

\section{Introduction}

Factorization theory studies the arithmetic structure of monoids and domains that are not factorial. A monoid is called a Krull monoid if it is a completely integrally closed Mori monoid. Moreover, a monoid is factorial if and only if it is a Krull monoid with trivial $t$-class group. Krull monoids are also a natural generalization of Dedekind domains, they are among the best understood objects in Factorization Theory and they posses several remarkable properties. For instance, it is possible to describe the arithmetic of Krull monoids purely in terms of their $t$-class group. More precisely, they allow a transfer homomorphism to a monoid of zero-sum sequences over their $t$-class group. Transfer homomorphisms allow to pull back the arithmetic properties from the target object to the source object. Thus, main parts of the arithmetic of Krull monoids can be studied in monoids of zero-sum sequences, where methods from Additive Combinatorics are available. It has since been asked whether there exist other types of monoids for which large parts of their arithmetic can be described by monoids of zero-sum sequences. In \cite{Ge16} the concept of abstract transfer Krull monoids was formally introduced and it provides a natural generalization of Krull monoids. It is well known that half-factorial monoids are another important class of transfer Krull monoids. But in general, transfer Krull monoids are neither Krull monoids nor half-factorial monoids.

\smallskip

The transfer Krull property has been studied in a variety of contexts. It was, for instance, investigated in the case of commutative unit-cancellative semigroups with identity, but also for noncommutative semigroups \cite{Ba-Ba-Go14, Ba-Sa20, Ba-Sm15, Ba-Sm18, Sm13, Sm19}. To mention one of the most striking results, let $R$ be a bounded hereditary noetherian prime ring. If every stably free right ideal is free, then there is a transfer homomorphism from the monoid of regular elements of $R$ to a Krull monoid. An overview on monoids (and domains) that allow resp. do not allow transfer homomorphisms to a Krull monoid can be found in \cite{Ge-Zh20}. We continue with a few more highlights to indicate the importance of transfer Krull monoids. In \cite{Ge-Ro20} it is proved that a strongly primary monoid is transfer Krull if and only if it is half-factorial. Moreover, a length-factorial monoid is transfer Krull if and only if it is Krull (\cite{Ge-Zh21}). In \cite{Ba-Ge-Re21} it is shown that the monoid of invertible ideals of a stable order in a Dedekind domain is transfer Krull if and only if it is half-factorial.

\smallskip
It is proved in \cite{Ge-Zh20} that a monoid is transfer Krull if and only if there is a Krull overmonoid for which the inclusion map (from the monoid into the Krull overmonoid) is a transfer homomorphism. Obviously, every Krull overmonoid of a monoid $H$ contains the root closure of $H$. The root closure is, therefore, the smallest possible candidate for a Krull overmonoid. The main purpose of this paper is to explore the impact of the root closure on the transfer Krull property. We complement the known literature with several new conditions that force a transfer Krull monoid to be Krull or half-factorial.

\smallskip

This paper consists of six sections including the introduction. In the next section, we introduce the most important notions and terminology and provide a series of elementary results involving transfer homomorphisms and transfer Krull monoids. Furthermore, we present more practicable (and simple) characterizations of the transfer Krull property for arbitrary monoids and $s$-noetherian monoids. The third section is devoted to the study of transfer Krull monoids whose root closure satisfies what we call property (U). It turns out, in particular, that a monoid $H$ whose root closure $\widetilde{H}$ satisfies property (U) is transfer Krull if and only if $\widetilde{H}$ is Krull and the inclusion map $H\hookrightarrow\widetilde{H}$ is a transfer homomorphism. Furthermore, we show that the valuation monoids are precisely the GCD-monoids which satisfy property (U). In Section 4 we discuss the effects of (half-)factoriality of the root closure, and show that a monoid whose root closure is half-factorial is transfer Krull if and only if it is half-factorial. Section 5 is devoted to the study of Cohen-Kaplansky domains and their generalizations. In particular, we rediscover and strengthen a result of \cite{Ba-Pi06} which states that a seminormal Cohen-Kaplansky domain is half-factorial and characterize the transfer Krull property for generalized Cohen-Kaplansky domains. In the last section we touch on the problem how transfer Krull monoids whose root closure is Krull can look like. We show that none of the aforementioned characterizations in this paper can be applied to monoids whose root closure is Krull. More precisely, we show that such simple descriptions cannot even be gathered for affine monoids (i.e., finitely generated monoids that are isomorphic to an additive submonoid of $\mathbb{Z}^n$ for some positive integer $n$).

\section{Notation and preliminaries}

Let $H$ be a commutative semigroup with identity. We say that $H$ is cancellative if for all $a,b,c\in H$ with $ac=bc$, it follows that $a=b$.

\bigskip
\begin{center}
{\it Throughout this paper, a monoid is always a commutative cancellative semigroup with identity.}
\end{center}
\bigskip

We denote by $\mathbb{Z}$ the set of integers, by $\mathbb{N}$ the set of positive integers and by $\mathbb{N}_0=\mathbb{N}\cup\{0\}$ the set of non-negative integers. For $r,s\in\mathbb{Z}$ let $[r,s]=\{x\in\mathbb{Z}\mid r\leq x\leq s\}$. Let $H$ be a (multiplicatively written) monoid and $K$ a quotient group of $H$. We let $H^{\times}$ denote the {\it group of units} of $H$ and we call $H$ {\it reduced} if $H^{\times}=\{1\}$. We denote by $H_{\rm red}=\{aH^{\times}\mid a\in H\}$ the {\it associated reduced monoid} of $H$. For $a,b\in H$ we set $a\mid_H b$ if $b=ac$ for some $c\in H$ and we set $a\simeq_H b$ if $a\mid_H b$ and $b\mid_H a$ (equivalently, $a=b\varepsilon$ for some $\varepsilon\in H^{\times}$). A subset $T$ of $H$ is called a {\it submonoid} of $H$ if $1\in T$ and $ab\in T$ for all $a,b\in T$. For a subset $S\subseteq H$, we let $[S]$ denote the smallest submonoid of $H$ containing $S$. If $S\subseteq K$, then let $\langle S\rangle$ denote the smallest subgroup of $K$ containing $S$. A submonoid of $K$ that contains $H$ is called an {\it overmonoid} of $H$. A subset $\mathfrak{a}\subseteq H$ is said to be an {\it $s$-ideal} of $H$ if $\mathfrak{a}H=\mathfrak{a}$. Let $\mathfrak{p}$ be an $s$-ideal of $H$. Then $\mathfrak{p}$ is said to be {\it prime} if $H\setminus\mathfrak{p}$ is a submonoid of $H$ and we denote by $s$-${\rm spec}(H)$ the set of all prime $s$-ideals of $H$. Moreover, an $s$-ideal $\mathfrak{a}$ of $H$ is called {\it $s$-finitely generated} if $\mathfrak{a}=EH$ for some finite subset $E\subseteq\mathfrak{a}$. By $\mathfrak{X}(H)$ we denote the set of minimal non-empty prime $s$-ideals of $H$. For $a\in H$ let $aH=\{ab\mid b\in H\}$ be the {\it principal ideal} generated by $a$. Clearly, every principal ideal is an $s$-ideal. We let $\mathcal{H}(H)$ denote the set of principal ideals of $H$.

\bigskip

The monoid $H$ is said to be {\it finitely generated} if there exists a finite subset $E\subseteq H$ such that $H=[E]$. We say that $H$ is {\it $s$-noetherian} if $H$ satisfies the ascending chain condition on $s$-ideals. Note that $H$ is $s$-noetherian if and only if every $s$-ideal of $H$ is $s$-finitely generated if and only if $H=[E\cup H^{\times}]$ for some finite subset $E\subseteq H$ if and only if $H_{\rm red}$ is finitely generated (see \cite[Proposition 2.7.4]{Ge-HK06}). Moreover, $H$ is finitely generated if and only if $H$ is $s$-noetherian and $H^{\times}$ is finitely generated. In particular, finitely generated monoids are $s$-noetherian and the converse holds for reduced monoids.

\bigskip

Let $u$ be a non-unit of $H$. Then $u$ is called an {\it atom} of $H$ if $u$ is not the product of two non-units of $H$. Moreover, $u$ is said to be a {\it prime element} of $H$ if $uH$ is a prime $s$-ideal of $H$. We denote by $\mathcal{A}(H)$ the set of atoms of $H$ and say that $H$ is {\it atomic} if every non-unit can be written as a finite product of atoms. For each non-unit $a\in H$, we let $\mathsf{L}_H(a)=\mathsf{L}(a)=\{k\in\mathbb{N}\mid a$ is a product of $k$ atoms of $H\}\subseteq\mathbb{N}$ be the {\it set of lengths} of $a$. Furthermore, we set $\mathsf{L}_H(a)=\mathsf{L}(a)=\{0\}$ for each $a\in H^{\times}$. An atomic monoid $H$ is said to be {\it factorial} if every atom of $H$ is a prime element and it is called {\it half-factorial} if $|\mathsf{L}(a)|=1$ for all $a\in H$. Observe that every factorial monoid is half-factorial.

We denote by
\begin{itemize}
\item $H^{\prime}=\{x\in K\mid$ there exists some $N\in\mathbb{N}$ such that $x^n\in H$ for all $n\geq N\}$ the {\it seminormal closure} of $H$ (also called the {\it seminormalization} of $H$), by
\item $\widetilde{H}=\{x\in K\mid x^N\in H$ for some $N\in\mathbb{N}\}$ the {\it root closure} of $H$, and by
\item $\widehat{H}=\{x\in K\mid$ there exists some $c\in H$, such that $cx^n\in H$ for all $n\in\mathbb{N}\}$ the {\it complete integral closure} of $H$.
\end{itemize}

We have $H\subseteq H^{\prime}\subseteq\widetilde{H}\subseteq\widehat{H}\subseteq K$. Furthermore, $H$ is said to be {\it seminormal} (resp., {\it root closed}, resp., {\it completely integrally closed}) if $H=H^{\prime}$ (resp., $H=\widetilde{H}$, resp., $H=\widehat{H}$). Let $A,B\subseteq K$ be subsets. We set $AB=\{ab\mid a\in A,b\in B\}$, $(A:B)=\{z\in K\mid zB\subseteq A\}$, $A^{-1}=(H:A)$, $A_v=(A^{-1})^{-1}$ and $A_t=\bigcup_{E\subseteq A,|E|<\infty} E_v$. If $A\subseteq H$, then $A$ is called a {\it $t$-ideal} (resp., {\it $v$-ideal}) of $H$ if $A_t=A$ (resp., {\it $A_v=A$}). A $t$-ideal $C$ of $H$ is said to be {\it $t$-invertible} if $(CC^{-1})_t=H$ and $C$ is called {\it $t$-finitely generated} if $C=E_t$ for some finite subset $E\subseteq C$. Note that every $t$-ideal is an $s$-ideal and every principal ideal is a $t$-invertible $t$-ideal. Let $\mathcal{C}_t(H)$ denote the $t$-class group of $H$. It measures how far $t$-invertible $t$-ideals are from being principal and it is trivial if and only if every $t$-invertible $t$-ideal is principal. For the precise definition of the $t$-class group we refer to \cite{HK98}. The monoid $H$ is said to be

\begin{itemize}
\item {\it Mori} if it satisfies the ascending chain condition on $t$-ideals,
\item {\it Krull} if it is a completely integrally closed Mori monoid,
\item {\it primary} if $H\ne H^{\times}$ and for all $a,b\in H\setminus H^{\times}$ there is $n\in\mathbb{N}$ such that $b^n\in aH$,
\item a {\it DVM} if $H$ is a primary monoid for which $H\setminus H^{\times}$ is a principal ideal of $H$, and
\item {\it finitely primary} if $H$ is a primary monoid, $(H:\widehat{H})\not=\emptyset$ and $\widehat{H}$ is factorial. If $H$ is finitely primary, then $|\mathfrak{X}(\widehat{H})|$ is called the {\it rank} of $H$ and each $\alpha\in\mathbb{N}$ for which $(\prod_{\mathfrak{p}\in\mathfrak{X}(\widehat{H})}\mathfrak{p})^{\alpha}\subseteq (H:\widehat{H})$, is called an {\it exponent} of $H$.
\end{itemize}

Note that $H$ is a Mori monoid if and only if $H$ satisfies that ascending chain condition on $v$-ideals. Also observe that the $t$-closure (and the $v$-closure) induce ideal systems in the sense of \cite{HK98}. All concepts involving $t$-ideals above (like $t$-invertibility and $t$-class group) can also be introduced in analogy for $v$-ideals. For our purposes the notion of $t$-ideals is the more useful notion. (This is the case, for instance, since the $t$-system is a finitary ideal system.) For more information on $t$-ideals, $v$-ideals and their relationships we refer to \cite{Ge-HK06, HK98}.

It is well known that a monoid is a DVM if and only if it is a primary Krull monoid if and only if it is a primary factorial monoid (cf. \cite[Theorem 2.3.8]{Ge-HK06}). Besides that, every DVM is a finitely primary monoid of rank one and exponent one. A monoid is factorial if and only if it is a Krull monoid with trivial $t$-class group and it is Krull if and only if every non-empty $t$-ideal is $t$-invertible. Furthermore, the root closure of an $s$-noetherian monoid is a Krull monoid.

Let $H$ and $B$ be monoids. A monoid homomorphism $\theta:H\to B$ is said to be a {\it transfer homomorphism} if the following two properties are satisfied.
\begin{enumerate}
\item[\textbf{(T1)}] $B=\theta(H) B^\times$ and $\theta^{-1}(B^\times)=H^\times$.
\item[\textbf{(T2)}] If $u\in H$, $b,c\in B$ and $\theta(u)=bc$, then there exist $v,w\in H$ and $\varepsilon,\eta\in B^{\times}$ such that $u=vw$,
		$\theta(v)=b\varepsilon$ and $\theta(w)=c\eta$.
\end{enumerate}

The monoid $H$ is said to be a {\it transfer Krull} monoid if there exists a Krull monoid $B$ and a transfer homomorphism $\theta:H\to B$. Since the identity homomorphism is trivially a transfer homomorphism, Krull monoids are transfer Krull, but transfer Krull monoids need neither be Mori, nor completely integrally closed. So far, the arithmetic of Krull monoids is very well understood and via transfer homomorphisms, the arithmetical properties of a Krull monoid $B$ are pulled back to the monoid $H$. For instance, transfer homomorphisms preserve the system of sets of lengths, whence all the invariants describing the structure of sets of lengths coincide. We refer the reader to the surveys \cite{Ge16, Ge-Zh20} for further details.

\begin{proposition}\label{proposition1} Let $H$ be a monoid with quotient group $K$. Then $H$ is a transfer Krull monoid if and only if there is a Krull monoid $T$ with $H\subseteq T\subseteq K$ such that the inclusion $H\hookrightarrow T$ is a transfer homomorphism. If this holds, then $K$ is the quotient group of $T$, $T=HT^{\times}$ and $T^{\times}\cap H=H^{\times}$.
\end{proposition}

\begin{proof} See \cite[Proposition 5.3]{Ge-Zh20}.
\end{proof}

In the following remark we collect several useful facts and provide a variety of situations in which a transfer Krull monoid is forced to be a Krull monoid or a half-factorial monoid.

\begin{remark}\label{remark1}
\begin{enumerate}
\item[(1)] Every half-factorial monoid is transfer Krull. Indeed, if $H$ is half-factorial, then $\theta:H\to (\mathbb{N}_0,+)$, $a\mapsto\max\mathsf{L}(a)$, is a transfer homomorphism.		
\item[(2)] Strongly primary monoids (i.e., primary monoids $H$ such that for each $x\in H$, there is an $n\in\mathbb{N}$ for which $(H\setminus H^{\times})^n\subseteq xH$) are transfer Krull if and only if they are half-factorial by \cite[Theorem 5.5]{Ge-Sc-Zh17}.
\item[(3)] Length-factorial monoids (i.e., atomic monoids $H$ such that for all $a\in H$ and $k\in\mathsf{L}(a)$, there is exactly one way (up to order and associates) to write $a$ as a product of $k$ atoms) are transfer Krull if and only if they are Krull by \cite[Corollary 1.5]{Ge-Zh21}.
\item[(4)] If $G$ is a finite group, then $\mathcal{B}({G})$ (i.e., the monoid of product-one sequences, cf. \cite[Definition 3.1]{Oh19}) is a reduced finitely generated monoid. Furthermore, $\mathcal{B}(G)$ is transfer Krull if and only if $\mathcal{B}(G)$ is Krull if and only if $G$ is abelian by \cite[Proposition 3.4]{Oh19}. For further recent results on the transfer Krull property for monoids of product-one sequences, we refer to \cite{Fa-Zh21, Ga-Li-Tr-Zh21}.
\item[(5)] If $R$ is a stable order in a Dedekind domain, then it follows from \cite[Theorem 5.10]{Ba-Ge-Re21} that the monoid of invertible ideals of $R$ (resp., the semigroup of nonzero ideals of $R$) is transfer Krull if and only if it is half-factorial.
\end{enumerate}
\end{remark}

Let $T$ be a monoid and let $H$ be a submonoid of $T$. We say that $H\subseteq T$ is {\it inert} if for all $x,y\in T$ with $xy\in H$, there is some $\varepsilon\in T^{\times}$ such that $x\varepsilon,y\varepsilon^{-1}\in H$. The concept of an inert extension was introduced in \cite{Co68} for ring extensions and studied in \cite{Mo21}. We adapt it for monoid extensions accordingly. Following the terminology of \cite{Ge-HK06}, $H\subseteq T$ is called {\it divisor-closed} if for all $x,y\in T$ with $xy\in H$, it follows that $x,y\in H$. Observe that $H\subseteq T$ is divisor-closed if and only if $H\subseteq T$ is inert and $T^{\times}=H^{\times}$.

\begin{lemma}\label{lemma1} Let $T$ be a monoid and let $H$ be a submonoid of $T$.
\begin{enumerate}
\item[(1)] $H\hookrightarrow T$ is a transfer homomorphism if and only if $T^{\times}\cap H=H^{\times}$, $T=HT^{\times}$ and $H\subseteq T$ is inert.
\item[(2)] If $T$ is an overmonoid of $H$, then $H\hookrightarrow T$ is a transfer homomorphism if and only if $T^{\times}\cap H=H^{\times}$ and $H\subseteq T$ is inert.
\item[(3)] $H\hookrightarrow\widetilde{H}$ is a transfer homomorphism if and only if $H\subseteq\widetilde{H}$ is inert.
\item[(4)] $H^{\prime}\hookrightarrow\widetilde{H}$ is a transfer homomorphism if and only if $H^{\prime}\subseteq\widetilde{H}$ is inert.
\end{enumerate}
\end{lemma}

\begin{proof} (1) By definition, $H\hookrightarrow T$ is a transfer homomorphism if and only if (a) $T^{\times}\cap H=H^{\times}$, (b) $T=HT^{\times}$ and (c) for all $a\in H$ and $x,y\in T$ such that $a=xy$, there are some $x^{\prime},y^{\prime}\in H$ and $\varepsilon,\eta\in T^{\times}$ such that $a=x^{\prime}y^{\prime}$, $x=x^{\prime}\eta$ and $y=y^{\prime}\varepsilon$.

($\Rightarrow$) Let $H\hookrightarrow T$ be a transfer homomorphism. It remains to show that $H\subseteq T$ is inert. Let $x,y\in T$ be such that $xy\in H$. Set $a=xy$. Then there are some $x^{\prime},y^{\prime}\in H$ and $\varepsilon,\eta\in T^{\times}$ such that $a=x^{\prime}y^{\prime}$, $x=x^{\prime}\eta$ and $y=y^{\prime}\varepsilon$. It follows that $x^{\prime}y^{\prime}=a=xy=x^{\prime}\eta y^{\prime}\varepsilon$, and thus $\varepsilon\eta=1$. Consequently, $\eta=\varepsilon^{-1}$, and hence $x\varepsilon=x^{\prime}\in H$ and $y\varepsilon^{-1}=y^{\prime}\in H$.

($\Leftarrow$) Let $T^{\times}\cap H=H^{\times}$, let $T=HT^{\times}$ and let $H\subseteq T$ be inert. It remains to show that for all $a\in H$ and $x,y\in T$ such that $a=xy$, there are some $x^{\prime},y^{\prime}\in H$ and $\varepsilon,\eta\in T^{\times}$ such that $a=x^{\prime}y^{\prime}$, $x=x^{\prime}\eta$ and $y=y^{\prime}\varepsilon$.

Let $a\in H$ and $x,y\in T$ be such that $a=xy$. Then $xy\in H$, and hence $x\varepsilon,y\varepsilon^{-1}\in H$ for some $\varepsilon\in T^{\times}$. Set $\eta=\varepsilon^{-1}$, $x^{\prime}=x\varepsilon$ and $y^{\prime}=y\varepsilon^{-1}$. Then $x^{\prime},y^{\prime}\in H$, $\varepsilon,\eta\in T^{\times}$, $a=x^{\prime}y^{\prime}$, $x=x^{\prime}\eta$ and $y=y^{\prime}\varepsilon$.

(2) Let $T$ be an overmonoid of $H$. By (1) it suffices to show that if $H\subseteq T$ is inert, then $T\subseteq HT^{\times}$. Let $H\subseteq T$ be inert and let $y\in T$. Since $T$ is an overmonoid of $H$, there is some $x\in H$ such that $xy\in H$. Therefore, there is some $\varepsilon\in T^{\times}$ such that $x\varepsilon,y\varepsilon^{-1}\in H$. We infer that $y=y\varepsilon^{-1}\varepsilon\in HT^{\times}$.

\medskip
(3) This is an immediate consequence of (2) and the fact that $\widetilde{H}$ is an overmonoid of $H$ and $\widetilde{H}^{\times}\cap H=H^{\times}$.

\medskip
(4) Observe that $\widetilde{H^{\prime}}=\widetilde{H}$, and hence $\widetilde{H}^{\times}\cap H^{\prime}={H^{\prime}}^{\times}$. Now this is an easy consequence of (2).
\end{proof}

\begin{proposition}\label{proposition2} Let $H$ be an $s$-noetherian monoid. The following statements are equivalent.
\begin{enumerate}
\item[(1)] $H$ is transfer Krull.
\item[(2)] There is some root closed overmonoid $T$ of $H$ such that $H\hookrightarrow T$ is a transfer homomorphism.
\item[(3)] There is some overmonoid $T$ of $H$ such that $T\subseteq\widetilde{T}$ is inert and $H\hookrightarrow T$ is a transfer homomorphism.
\end{enumerate}
\end{proposition}

\begin{proof} (1) $\Rightarrow$ (2): Let $H$ be transfer Krull. It follows from Proposition~\ref{proposition1} that there is some overmonoid $T$ of $H$ such that $T$ is Krull monoid and $H\hookrightarrow T$ is a transfer homomorphism. Since $T$ is a Krull monoid, we have that $T$ is root closed.

\medskip
(2) $\Rightarrow$ (1): Let $T$ be a root closed overmonoid of $H$ such that $H\hookrightarrow T$ is a transfer homomorphism. We infer by Lemma~\ref{lemma1}(2) that $H\hookrightarrow T$ is a transfer homomorphism, and thus $T=HT^{\times}$. There is some finite subset $E\subseteq H$ such that $H=[E\cup H^{\times}]$. Since $T=HT^{\times}$, it follows that $T=[E\cup T^{\times}]$. Therefore, $T$ is a root closed $s$-noetherian monoid, and thus $T$ is a Krull monoid by \cite[Theorem 2.7.14]{Ge-HK06}. Since $H\hookrightarrow T$ is a transfer homomorphism, we have that $H$ is transfer Krull.

\medskip
(2) $\Rightarrow$ (3): This is obvious.

\medskip
(3) $\Rightarrow$ (2): Let $T$ be an overmonoid of $H$ such that $T\subseteq\widetilde{T}$ is inert and $H\hookrightarrow T$ is a transfer homomorphism. Then $T\hookrightarrow\widetilde{T}$ is a transfer homomorphism by Lemma~\ref{lemma1}(3), and hence $H\hookrightarrow\widetilde{T}$ is a transfer homomorphism. Now the statement follows, since $\widetilde{T}$ is a root closed overmonoid of $H$.
\end{proof}

\begin{proposition}\label{proposition3} Let $H$ be a monoid with quotient group $K$.
\begin{enumerate}
\item[(1)] If $H\subseteq\widetilde{H}$ is inert and $\widetilde{H}=\{x\in K\mid x^k\in H\}$ for some $k\in\mathbb{N}$, then $H^{\prime}\subseteq\widetilde{H}$ is inert.
\item[(2)] If $H$ is $s$-noetherian, then $H^{\prime}$ is $s$-noetherian.
\item[(3)] If $H$ is $s$-noetherian and $H\subseteq\widetilde{H}$ is inert, then $H^{\prime}$ is $s$-noetherian and $H^{\prime}\subseteq\widetilde{H}$ is inert.
\end{enumerate}
\end{proposition}

\begin{proof} (1) Let $H\subseteq\widetilde{H}$ be inert and let $k\in\mathbb{N}$ be such that $\widetilde{H}=\{x\in K\mid x^k\in H\}$. It remains to show that for all $x,y\in\widetilde{H}$ with $xy\in H^{\prime}$, there is some $\varepsilon\in\widetilde{H}^{\times}$ such that $x\varepsilon,y\varepsilon^{-1}\in H^{\prime}$. Let $x,y\in\widetilde{H}$ be such that $xy\in H^{\prime}$. Then there is some $N\in\mathbb{N}$ such that $(xy)^n\in H$ for each $n\in\mathbb{N}_{\geq N}$, and hence $(xy)^{kN+1}\in H$. Observe that there is some $\varepsilon\in\widetilde{H}^{\times}$ such that $x^{kN+1}\varepsilon,y^{kN+1}\varepsilon^{-1}\in H$. We have that $(x\varepsilon)^{kN+1}=x^{kN+1}\varepsilon (\varepsilon^N)^k\in H$, $(x\varepsilon)^{2kN+1}=x^{kN+1}\varepsilon (x^N\varepsilon^{2N})^k\in H$, $(y\varepsilon^{-1})^{kN+1}=y^{kN+1}\varepsilon^{-1} (\varepsilon^{-N})^k\in H$ and $(y\varepsilon^{-1})^{2kN+1}=y^{kN+1}\varepsilon^{-1} (y^N\varepsilon^{-2N})^k\in H$. Since $kN+1$ and $2kN+1$ are relatively prime, we infer that $x\varepsilon,y\varepsilon^{-1}\in H^{\prime}$.

\medskip
(2) Let $H$ be $s$-noetherian. Then $(H:\widetilde{H})\not=\emptyset$ by \cite[Propositions 2.7.4 and 2.7.11 and Theorem 2.7.13]{Ge-HK06}. Therefore, $(H:H^{\prime})\not=\emptyset$. Let $(\mathfrak{a}_i)_{i\in\mathbb{N}}$ be an ascending chain of $s$-ideals of $H^{\prime}$. If $x\in (H:H^{\prime})$, then $(x\mathfrak{a}_i)_{i\in\mathbb{N}}$ is an ascending chain of $s$-ideals of $H$, and hence there is some $N\in\mathbb{N}$ such that $x\mathfrak{a}_i=x\mathfrak{a}_N$ for each integer $i\geq N$. Consequently, $\mathfrak{a}_i=\mathfrak{a}_N$ for each integer $i\geq N$.

\medskip
(3) Let $H$ be $s$-noetherian and let $H\subseteq\widetilde{H}$ be inert. It is an immediate consequence of \cite[Propositions 2.7.4 and 2.7.11 and Theorem 2.7.13]{Ge-HK06} that there is a finite subset $E\subseteq\widetilde{H}$ such that $\widetilde{H}=[E\cup H^{\times}]$ (note that $\widetilde{H}/H^{\times}$ is finitely generated). Since $E$ is finite, there is clearly some $k\in\mathbb{N}$ such that $e^k\in H$ for each $e\in E$. It is straightforward to show that $\widetilde{H}=\{x\in K\mid x^k\in H\}$. Now the statement follows from (1) and (2).
\end{proof}

Next we provide a generalization of \cite[Proposition 3.7.5.1]{Ge-HK06}. For the definition of ideal systems, we refer to \cite{HK98}. For a monoid $H$ and an ideal system $r$ on $H$ let $r$-$\max(H)$ denote the {\it set of $r$-maximal $r$-ideals} of $H$. Note that the $s$-system, the $t$-system and the $v$-system (they are induced by the $s$-ideals, $t$-ideals and $v$-ideals, respectively) are important examples of ideal systems on monoids.

\begin{proposition}\label{proposition4} Let $H$ be a monoid, let $r$ be an ideal system on $H$ and let $T$ be an overmonoid of $H$ such that $T^{\times}\cap H=H^{\times}$ and $T=HT^{\times}$. If $(H:T)\in r$-$\max(H)$, then $H\hookrightarrow T$ is a transfer homomorphism.
\end{proposition}

\begin{proof} Let $(H:T)\in r$-$\max(H)$. By Lemma~\ref{lemma1}(2), it suffices to show that for all $x,y\in T$ with $xy\in H$ there is some $\varepsilon\in T^{\times}$ such that $x\varepsilon,y\varepsilon^{-1}\in H$. Let $x,y\in T$ be such that $xy\in H$. There are some $\alpha,\beta\in T^{\times}$ and $u,v\in H$ such that $x=u\alpha$ and $y=v\beta$.

Case 1: $u\in (H:T)$. Then $x\beta=u\alpha\beta\in H$ and $y\beta^{-1}=v\in H$.

Case 2: $u\not\in (H:T)$. Then $(uH\cup (H:T))_r=H$. There is some $t\in H$ such that $ty\alpha\in H$. It follows that $ty\alpha\in ty\alpha H=ty\alpha (uH\cup (H:T))_r=(ty\alpha uH\cup ty\alpha (H:T))_r=(txyH\cup ty\alpha (H:T))_r\subseteq (txyH\cup tH)_r=tH$, and hence $y\alpha\in H$. Finally, observe that $x\alpha^{-1}=u\in H$.
\end{proof}

\begin{proposition}\label{proposition5} Let $T$ be a monoid and let $H\subseteq T$ be a submonoid.
\begin{enumerate}
\item[(1)] If $T^{\times}\cap H=H^{\times}$ and $T=HT^{\times}$, then $\mathcal{A}(T)\subseteq\{u\varepsilon\mid u\in\mathcal{A}(H),\varepsilon\in T^{\times}\}$.
\item[(2)] If $H\hookrightarrow T$ is a transfer homomorphism, then $\mathcal{A}(T)=\{u\varepsilon\mid u\in\mathcal{A}(H),\varepsilon\in T^{\times}\}$.
\item[(3)] If $T$ is atomic and $\mathcal{A}(T)\subseteq\{u\varepsilon\mid u\in\mathcal{A}(H),\varepsilon\in T^{\times}\}$, then $T=HT^{\times}$.
\item[(4)] If $H$ is atomic and $\mathcal{A}(H)\subseteq\mathcal{A}(T)$, then $T^{\times}\cap H=H^{\times}$.
\end{enumerate}
\end{proposition}

\begin{proof} (1) Let $T^{\times}\cap H=H^{\times}$ and $T=HT^{\times}$. Let $v\in\mathcal{A}(T)$. There are some $u\in H$ and $\varepsilon\in T^{\times}$ such that $v=u\varepsilon$. It remains to show that $u\in\mathcal{A}(H)$. Clearly, $u\not\in H^{\times}$ (for if $u\in H^{\times}$, then $v\in T^{\times}$). Now let $a,b\in H$ be such that $u=ab$. Then $v=a\varepsilon b$, and hence $a\varepsilon\in T^{\times}$ or $b\in T^{\times}$, so $a\in T^{\times}\cap H=H^{\times}$ or $b\in T^{\times}\cap H=H^{\times}$.

\medskip
(2) Let $H\hookrightarrow T$ be a transfer homomorphism. It is well known that $\mathcal{A}(H)\subseteq\mathcal{A}(T)$ (e.g. see \cite[Proposition 3.2.3.2]{Ge-HK06}). Now the statement follows from (1).

\medskip
(3) Let $T$ be atomic and let $\mathcal{A}(T)\subseteq\{u\varepsilon\mid u\in\mathcal{A}(H),\varepsilon\in T^{\times}\}$. Let $x\in T$. Without restriction let $x\not\in T^{\times}$. Then $x=\prod_{i=1}^n u_i$ for some $n\in\mathbb{N}$ and atoms $u_i$ of $T$. For each $i\in [1,n]$ there are some $v_i\in\mathcal{A}(H)$ and $\varepsilon_i\in T^{\times}$ such that $u_i=v_i\varepsilon_i$. We infer that $x=(\prod_{i=1}^n v_i)(\prod_{i=1}^n\varepsilon_i)\in HT^{\times}$.

\medskip
(4) Let $H$ be atomic and let $\mathcal{A}(H)\subseteq\mathcal{A}(T)$. Let $x\in T^{\times}\cap H$. Assume that $x\not\in H^{\times}$. Then $x=ua$ for some $u\in\mathcal{A}(H)$ and $a\in H$. Since $x\in T^{\times}$, we infer that $u\in T^{\times}$, which contradicts the fact that $u\in\mathcal{A}(T)$.
\end{proof}

\begin{remark}\label{remark2} Let $H\subseteq L\subseteq T$ be monoids such that $H\hookrightarrow T$ is a transfer homomorphism.
\begin{enumerate}
\item[(1)] If $T^{\times}\cap L=L^{\times}$ and $L=HL^{\times}$, then $L\hookrightarrow T$ is a transfer homomorphism.
\item[(2)] If $\widetilde{H}\subseteq T$ and $\widetilde{H}=H\widetilde{H}^{\times}$, then $\widetilde{H}\hookrightarrow T$ is a transfer homomorphism.
\end{enumerate}
\end{remark}

\begin{proof} (1) Let $T^{\times}\cap L=L^{\times}$ and $L=HL^{\times}$. Since $T=HT^{\times}$, we infer that $T=LT^{\times}$. By Lemma~\ref{lemma1}(1), it remains to show that $L\subseteq T$ is inert. Let $x,y\in T$ be such that $xy\in L$. There are some $z\in H$ and $\eta\in L^{\times}$ such that $xy=z\eta$. We have that $x\eta^{-1}y=z\in H$ and $x\eta^{-1},y\in T$. Therefore, $x\eta^{-1}\varepsilon,y\varepsilon^{-1}\in H$ for some $\varepsilon\in T^{\times}$. It follows that $x\varepsilon,y\varepsilon^{-1}\in L$.

(2) This is an easy consequence of (1) and the fact that $T^{\times}\cap\widetilde{H}=\widetilde{H}^{\times}$. (For more details see the proof of Theorem~\ref{theorem1} below.)
\end{proof}

\section{Transfer Krull monoids and property (U)}

Let $H$ be a monoid with quotient group $K$. Then $H$ is called a {\it valuation monoid} if for each $x\in K$ we have that $x\in H$ or $x^{-1}\in H$. Observe that $H$ is a DVM if and only if $H$ is an atomic valuation monoid and $H\not=H^{\times}$. For a submonoid $S$ of $H$ let $S^{-1}H=\{s^{-1}x\mid s\in S,x\in H\}$ which is clearly an overmonoid of $H$. Note that if $S\subseteq H$ is a submonoid and $T=S^{-1}H$, then $T=HT^{\times}$. For each prime $s$-ideal $\mathfrak{p}$ of $H$ and each overmonoid $T$ of $H$ we set $H_{\mathfrak{p}}=(H\setminus\mathfrak{p})^{-1}H$ and $T_{\mathfrak{p}}=(H\setminus\mathfrak{p})^{-1}T$.

Next we introduce an ``ad-hoc'' property which is motivated by the concept of QR-domains (i.e., domains for which every overring is a quotient overring). (QR-domains will be discussed in Section 5.) We say that $H$ satisfies {\it property (U)} if for each overmonoid $T$ of $H$ with $T=HT^{\times}$, there is some submonoid $S\subseteq H$ such that $T=S^{-1}H$. Note that if $H$ is a valuation monoid, then for each overmonoid $T$ of $H$, there is some submonoid $S\subseteq H$ such that $T=S^{-1}H$. In particular, every valuation monoid satisfies property (U).

(We show that each overmonoid of a valuation monoid $H$ is of the form $S^{-1}H$ for some submonoid $S$ of $H$. Let $H$ be a valuation monoid, let $T$ be an overmonoid of $H$ and set $S=\{x\in H\mid x^{-1}\in T\}$. Then $S\subseteq H$ is a submonoid and $S^{-1}H\subseteq T$. It suffices to show that $T\subseteq S^{-1}H$. Let $z\in T$. Then $z\in H$ or $z^{-1}\in H$. If $z\in H$, then $z\in S^{-1}H$. If $z^{-1}\in H$, then $z^{-1}\in S$, and hence $z=(z^{-1})^{-1}\in S^{-1}\subseteq S^{-1}H$.)

\begin{theorem}\label{theorem1} Let $H$ be a monoid such that $\widetilde{H}$ satisfies property (U).
\begin{enumerate}
\item[(1)] $H$ is transfer Krull if and only if $\widetilde{H}$ is a Krull monoid and $H\subseteq\widetilde{H}$ is inert.
\item[(2)] The following statements are equivalent.
\begin{enumerate}
\item[(a)] $H$ is half-factorial.
\item[(b)] $\widetilde{H}$ is half-factorial and $H$ is transfer Krull.
\item[(c)] $\widetilde{H}$ is half-factorial and $H\subseteq\widetilde{H}$ is inert.
\end{enumerate}
\end{enumerate}
\end{theorem}

\begin{proof} (1) Clearly, if $\widetilde{H}$ is a Krull monoid and $H\subseteq\widetilde{H}$ is inert, then $H$ is transfer Krull by Lemma~\ref{lemma1}(3). Now let $H$ be transfer Krull. By Proposition~\ref{proposition1} there is some overmonoid $T$ of $H$ which is a Krull monoid such that $H\hookrightarrow T$ is a transfer homomorphism. Note that $T^{\times}\cap H=H^{\times}$ and $T=HT^{\times}$. Since $H\subseteq T$ and $T$ is Krull, we have that $\widetilde{H}\subseteq\widetilde{T}=T$, and thus $T=HT^{\times}\subseteq\widetilde{H}T^{\times}\subseteq T$. Therefore, $T=\widetilde{H}T^{\times}$, and hence $T=S^{-1}\widetilde{H}$ for some submonoid $S\subseteq\widetilde{H}$.

Next we show that $T^{\times}\cap\widetilde{H}=\widetilde{H}^{\times}$. ($\subseteq$) Let $x\in T^{\times}\cap\widetilde{H}$. There is some $k\in\mathbb{N}$ such that $x^k\in H$. We infer that $x^k\in T^{\times}\cap H=H^{\times}\subseteq\widetilde{H}^{\times}$. Since $x\in\widetilde{H}$, it follows that $x\in\widetilde{H}^{\times}$. ($\supseteq$) This is clearly satisfied.

Observe that $S\subseteq T^{\times}$. Consequently, $S\subseteq T^{\times}\cap\widetilde{H}=\widetilde{H}^{\times}$, and thus $T=S^{-1}\widetilde{H}=\widetilde{H}$. Therefore, $\widetilde{H}$ is a Krull monoid and $H\subseteq\widetilde{H}$ is inert by Lemma~\ref{lemma1}(3).

\medskip
(2) (a) $\Rightarrow$ (b): It is well known that $H$ is transfer Krull. By (1) and Lemma~\ref{lemma1}(2) we have that $H\hookrightarrow\widetilde{H}$ is a transfer homomorphism, and thus $\widetilde{H}$ is half-factorial by \cite[Proposition 3.2.3]{Ge-HK06}.

(b) $\Rightarrow$ (c): This follows from (1).

(c) $\Rightarrow$ (a): Note that $H\hookrightarrow\widetilde{H}$ is a transfer homomorphism by Lemma~\ref{lemma1}(3). Now the statement follows from \cite[Proposition 3.2.3]{Ge-HK06}.
\end{proof}

\begin{corollary}\label{corollary1} Let $H$ be a monoid such that $\widetilde{H}$ is half-factorial and satisfies property (U). The following statements are equivalent.
\begin{enumerate}
\item[(1)] $H$ is transfer Krull.
\item[(2)] $H$ is half-factorial.
\item[(3)] $H\subseteq\widetilde{H}$ is inert.
\end{enumerate}
If these equivalent conditions are satisfied, then $\widetilde{H}$ is a Krull monoid.
\end{corollary}

\begin{proof} This is an immediate consequence of Theorem~\ref{theorem1}.
\end{proof}

\begin{corollary}\label{corollary2} Let $H$ be a monoid such that $\widetilde{H}$ is a DVM (e.g. $H$ is an $s$-noetherian finitely primary monoid of rank one). The following statements are equivalent.
\begin{enumerate}
\item[(1)] $H$ is transfer Krull.
\item[(2)] $H$ is half-factorial.
\item[(3)] $H\subseteq\widetilde{H}$ is inert.
\end{enumerate}
\end{corollary}

\begin{proof} This is an immediate consequence of Corollary~\ref{corollary1}, since DVMs are half-factorial and valuation monoids satisfy property (U).
\end{proof}

\begin{corollary}\label{corollary3} Let $H$ be a monoid such that $\widetilde{H}$ is half-factorial and satisfies property (U). Then $H$ is a Krull monoid if and only if $H$ is root closed.
\end{corollary}

\begin{proof} Clearly, if $H$ is a Krull monoid, then $H$ is root closed. Now let $H$ be root closed. Then $H=\widetilde{H}$ is half-factorial. Therefore, $H=\widetilde{H}$ is a Krull monoid by Corollary~\ref{corollary1}.
\end{proof}

Let $H$ be a monoid. Then $H$ is called a {\it GCD-monoid} if each two elements $a,b\in H$ have a greatest common divisor (i.e., there is some $t\in H$ such that $t\mid_H a$ and $t\mid_H b$ and for all $s\in H$ with $s\mid_H a$ and $s\mid_H b$, it follows that $s\mid_H t$). Note that $H$ is a GCD-monoid if and only if every $t$-finitely generated $t$-ideal of $H$ is principal by \cite[Theorem 11.5(iii)]{HK98} (since the $v$-finitely generated $v$-ideals are precisely the $t$-finitely generated $t$-ideals). It is well known that factorial monoids and valuation monoids are GCD-monoids. For a thorough introduction to GCD-monoids, we refer to \cite{HK98}.

\begin{proposition}\label{proposition6} Let $H$ be a monoid. Then $H$ is a valuation monoid if and only if $H$ is a GCD-monoid which satisfies property (U).
\end{proposition}

\begin{proof} Clearly, every valuation monoid is a GCD-monoid which satisfies property (U). Now let $H$ be a GCD-monoid which satisfies property (U). Since $H$ is a GCD-monoid, it suffices to show that for all relatively prime $x,y\in H$, it follows that $x\in H^{\times}$ or $y\in H^{\times}$. Let $x,y\in H$ be relatively prime (i.e., for each $t\in H$ with $t\mid_H x$ and $t\mid_H y$, we have that $t\in H^{\times}$). Set $T=\{a(\frac{x}{y})^k\mid a\in H,k\in\mathbb{Z}\}$. Observe that $T$ is an overmonoid of $H$ such that $T=HT^{\times}$. Consequently, there is some submonoid $S\subseteq H$ such that $T=S^{-1}H$. Since $\frac{x}{y}\in T$, there are some $u\in H$ and $s\in S$ such that $\frac{x}{y}=\frac{u}{s}$. We infer that $xs=yu$. Since $H$ is a GCD-monoid and $x$ and $y$ are relatively prime, there is some $w\in H$ such that $s=yw$ by \cite[Proposition 10.2(ii)]{HK98}. Therefore, $y^{-1}=\frac{w}{s}\in S^{-1}H=T$. It follows that $y^{-1}=a(\frac{x}{y})^k$ for some $a\in H$ and $k\in\mathbb{Z}$. If $k\geq 1$, then $y^{k-1}=ax^k$, and hence $x\in H^{\times}$ (since $x\mid_H y^{k-1}$, $H$ is a GCD-monoid and $x$ and $y$ are relatively prime). Now let $k\leq 0$ and set $n=-k$. Note that $n\in\mathbb{N}_0$ and $x^n=ay^{n+1}$. It follows that $y\in H^{\times}$ (since $y\mid_H x^n$, $H$ is a GCD-monoid and $x$ and $y$ are relatively prime).
\end{proof}

\section{Transfer Krull monoids with (half-)factorial root closure}

In this section we investigate when a monoid whose root closure is (half-)factorial is a transfer Krull monoid. As a consequence we characterize when a weakly factorial monoid whose root closure is Krull is transfer Krull.

\begin{lemma}\label{lemma2} Let $H$ be a monoid.
\begin{enumerate}
\item[(1)] $\mathcal{A}(H)\subseteq\mathcal{A}(\widetilde{H})$ if and only if there is some overmonoid $T$ of $\widetilde{H}$ such that $\mathcal{A}(H)\subseteq\mathcal{A}(T)$ and $T^{\times}\cap H=H^{\times}$.
\item[(2)] If $H$ is atomic, then $\mathcal{A}(H)\subseteq\mathcal{A}(\widetilde{H})$ if and only if there is some overmonoid $T$ of $\widetilde{H}$ such that $\mathcal{A}(H)\subseteq\mathcal{A}(T)$.
\item[(3)] If $H$ is transfer Krull, then $\mathcal{A}(H)\subseteq\mathcal{A}(\widetilde{H})$.
\end{enumerate}
\end{lemma}

\begin{proof} (1) First let $\mathcal{A}(H)\subseteq\mathcal{A}(\widetilde{H})$. Set $T=\widetilde{H}$. It is obvious that $T^{\times}\cap H=H^{\times}$.

Now let $T$ be an overmonoid of $\widetilde{H}$ such that $\mathcal{A}(H)\subseteq\mathcal{A}(T)$ and $T^{\times}\cap H=H^{\times}$. Observe that $T^{\times}\cap\widetilde{H}=\widetilde{H}^{\times}$. (If $x\in T^{\times}\cap\widetilde{H}$, then $x^k\in H$ for some $k\in\mathbb{N}$, and hence $x^k\in T^{\times}\cap H=H^{\times}\subseteq\widetilde{H}^{\times}$, so $x\in\widetilde{H}^{\times}$.)

Let $u\in\mathcal{A}(H)$. Since $u\not\in H^{\times}$ (and $\widetilde{H}^{\times}\cap H=H^{\times}$), we have that $u\not\in\widetilde{H}^{\times}$. Now let $a,b\in\widetilde{H}$ be such that $u=ab$. Then $a,b\in T$. Since $u\in\mathcal{A}(T)$, we infer that $a\in T^{\times}\cap\widetilde{H}=\widetilde{H}^{\times}$ or $b\in T^{\times}\cap\widetilde{H}=\widetilde{H}^{\times}$.

\medskip
(2) This is an immediate consequence of (1) and Proposition~\ref{proposition5}(4).

\medskip
(3) This follows from (1) and Propositions~\ref{proposition1} and~\ref{proposition5}.
\end{proof}

The problem whether an atom of a half-factorial domain is again an atom of certain overrings has already been studied (e.g. see \cite[Proposition 2.2]{Pi03}). Observe that Lemma~\ref{lemma2}(3) is a result of similar type for monoids. Next we present the first main result of this section. In Section 6 we provide examples of half-factorial monoids whose root closure is also half-factorial.

\begin{theorem}\label{theorem2} Let $H$ be a monoid.
\begin{enumerate}
\item[(1)] If $\widetilde{H}$ is half-factorial, then the following statements are equivalent.
\begin{enumerate}
\item[(a)] $H$ is transfer Krull.
\item[(b)] There is some overmonoid $T$ of $H$ such that $T$ is Krull and $\mathcal{A}(T)=\{u\varepsilon\mid u\in\mathcal{A}(H),\varepsilon\in T^{\times}\}$.
\item[(c)] There is some overmonoid $T$ of $\widetilde{H}$ such that $\mathcal{A}(H)\subseteq\mathcal{A}(T)$.
\item[(d)] $\mathcal{A}(H)\subseteq\mathcal{A}(\widetilde{H})$.
\item[(e)] $H$ is half-factorial.
\end{enumerate}
\item[(2)] If $\widetilde{H}$ is factorial, then $H$ is transfer Krull if and only if $\mathcal{A}(\widetilde{H})=\{u\varepsilon\mid u\in\mathcal{A}(H),\varepsilon\in\widetilde{H}^{\times}\}$.
\end{enumerate}
\end{theorem}

\begin{proof} (1) Let $\widetilde{H}$ be half-factorial. Since $\widetilde{H}^{\times}\cap H=H^{\times}$, we have by \cite[Corollary 1.3.3]{Ge-HK06} that $H$ is atomic.

(a) $\Rightarrow$ (b): This follows from Propositions~\ref{proposition1} and~\ref{proposition5}(2).

(b) $\Rightarrow$ (c): Since $T$ is root closed, we have that $\widetilde{H}\subseteq\widetilde{T}=T$, and hence $T$ is an overmonoid of $\widetilde{H}$.

(c) $\Rightarrow$ (d): This is an immediate consequence of Lemma~\ref{lemma2}(2).

(d) $\Rightarrow$ (e): Let $a\in H$ and let $k,\ell\in\mathsf{L}(a)$. Then $a$ is the product of $k$ atoms of $H$ and $a$ is the product of $\ell$ atoms of $H$. Consequently, $a$ is the product of $k$ atoms of $\widetilde{H}$ and $a$ is the product of $\ell$ atoms of $\widetilde{H}$. Since $\widetilde{H}$ is half-factorial, we have that $k=\ell$.

(e) $\Rightarrow$ (a): This is clear.

\medskip
(2) Let $\widetilde{H}$ be factorial. By (1) it remains to show that if $H$ is transfer Krull, then $\mathcal{A}(\widetilde{H})\subseteq\{u\varepsilon\mid u\in\mathcal{A}(H),\varepsilon\in\widetilde{H}^{\times}\}$. Let $H$ be transfer Krull and let $v\in\mathcal{A}(\widetilde{H})$. Set $Q=v\widetilde{H}$ and $P=Q\cap H$. Then $Q\in\mathfrak{X}(\widetilde{H})$ (since $\widetilde{H}$ is factorial) and $P\in\mathfrak{X}(H)$ by \cite[Proposition 5(b)]{Ch-HK-K02}. Since $H$ is atomic, there is some $u\in\mathcal{A}(H)\cap P$. We infer that $u\in\mathcal{A}(\widetilde{H})\cap Q$. This implies that $u\widetilde{H}=Q=v\widetilde{H}$, and thus $v=u\varepsilon$ for some $\varepsilon\in\widetilde{H}^{\times}$.
\end{proof}

Let $H$ be a monoid. A non-unit $a$ of $H$ is called {\it primary} if for all $b,c\in H$ with $a\mid_H bc$ and $a\nmid_H b$, there is some $n\in\mathbb{N}$ such that $a\mid_H c^n$. Moreover, $H$ is said to be {\it weakly factorial} if every non-unit of $H$ is a finite product of primary elements of $H$. In what follows, we freely use coproducts. Their precise definition can be found in \cite{Ge-HK06}.

\begin{remark}\label{remark3} Let $H$ be a weakly factorial monoid. Then $H$ is half-factorial if and only if $H_P$ is half-factorial for each $P\in\mathfrak{X}(H)$.
\end{remark}

\begin{proof} By \cite[Theorem 22.5(ii)]{HK98}, we have that the $t$-dimension of $H$ is at most one (i.e., $t$-$\max(H)\subseteq\mathfrak{X}(H)$). It is easy to see that $\bigcap_{P\in\mathfrak{X}(H)} xH_P=xH$ for each $x\in H$ (cf. \cite[Theorem 7.4]{HK98}). Also note that $xH_P\cap H$ is a principal ideal of $H$ by \cite[Exercise 5(i), page 258]{HK98} for all $x\in H$ and $P\in\mathfrak{X}(H)$. It is now straightforward to show that $\varphi:\mathcal{H}(H)\rightarrow\coprod_{P\in\mathfrak{X}(H)}\mathcal{H}(H_P)$ defined by $\varphi(xH)=(xH_P)_{P\in\mathfrak{X}(H)}$ for each $x\in H$ is a monoid isomorphism. It follows from \cite[Proposition 1.2.11.2]{Ge-HK06} that $H$ is half-factorial if and only if $H_{\rm red}\cong\mathcal{H}(H)$ is half-factorial if and only if $\coprod_{P\in\mathfrak{X}(H)} (H_P)_{\rm red}\cong\coprod_{P\in\mathfrak{X}(H)}\mathcal{H}(H_P)$ is half-factorial if and only if $(H_P)_{\rm red}$ is half-factorial for each $P\in\mathfrak{X}(H)$ if and only if $H_P$ is half-factorial for each $P\in\mathfrak{X}(H)$.
\end{proof}

Now we provide the second main theorem of this section. In Section 5 we introduce and discuss the concept of generalized Cohen-Kaplansky domain. The monoids of nonzero elements of these domains are among the most important examples of weakly factorial monoids whose root closure is a Krull monoid. In Section 6 it will become clear that even if one modestly weakens the weakly factorial property in the next result that the provided conditions no longer characterize the transfer Krull property.

\begin{theorem}\label{theorem3} Let $H$ be a weakly factorial monoid such that $\widetilde{H}$ is a Krull monoid. The following statements are equivalent.
\begin{enumerate}
\item[(1)] $H$ is transfer Krull.
\item[(2)] For each $P\in\mathfrak{X}(H)$, $H_P$ is transfer Krull.
\item[(3)] For each $P\in\mathfrak{X}(H)$, $H_P$ is half-factorial.
\item[(4)] $H$ is half-factorial.
\item[(5)] $\mathcal{A}(\widetilde{H})=\{u\varepsilon\mid u\in\mathcal{A}(H),\varepsilon\in\widetilde{H}^{\times}\}$.
\item[(6)] $\mathcal{A}(H)\subseteq\mathcal{A}(\widetilde{H})$.
\end{enumerate}
In addition, if $H$ is seminormal, then these equivalent conditions are satisfied.
\end{theorem}

\begin{proof} It follows from \cite[Exercise 5, page 258]{HK98} that the $t$-class group of $H$ is trivial, and hence the $t$-class group of $\widetilde{H}$ is trivial by \cite[Proposition 8]{Ch-HK-K02}. Therefore, $\widetilde{H}$ is factorial by \cite[Corollary 2.3.13]{Ge-HK06}. Next we show that $\widetilde{H_P}$ is a DVM for each $P\in\mathfrak{X}(H)$. Let $P\in\mathfrak{X}(H)$. There is some $Q\in\mathfrak{X}(\widetilde{H})$ such that $Q\cap H=P$ by \cite[Proposition 5(b)]{Ch-HK-K02}. Clearly, $\widetilde{H}_Q$ is a DVM by \cite[Theorem 2.3.11]{Ge-HK06}. We show that $\widetilde{H_P}=\widetilde{H}_Q$. Since $Q\cap H=P$, we have that $\widetilde{H_P}=(H\setminus P)^{-1}\widetilde{H}\subseteq (\widetilde{H}\setminus Q)^{-1}\widetilde{H}=\widetilde{H}_Q$. It remains to show that $(\widetilde{H}\setminus Q)^{-1}\subseteq\widetilde{H_P}$. Let $x\in\widetilde{H}\setminus Q$. There is some $k\in\mathbb{N}$ such that $x^k\in H$. Observe that $x^k\not\in P$. (If $x^k\in P$, then $x^k\in Q$, and thus $x\in Q$.) Therefore, $x^{-k}\in (H\setminus P)^{-1}\subseteq H_P$, and hence $x^{-1}\in\widetilde{H_P}$.

\medskip
(1) $\Leftrightarrow$ (4) $\Leftrightarrow$ (5) $\Leftrightarrow$ (6): This is an immediate consequence of Theorem~\ref{theorem2}.

\medskip
(2) $\Leftrightarrow$ (3): Let $P\in\mathfrak{X}(H)$. Since $\widetilde{H_P}$ is a DVM, we infer by Corollary~\ref{corollary2} (or by Theorem~\ref{theorem2}(1)) that $H_P$ is transfer Krull if and only if $H_P$ is half-factorial.

\medskip
(3) $\Leftrightarrow$ (4): This follows from Remark~\ref{remark3}.

\medskip
Now let $H$ be seminormal and let $P\in\mathfrak{X}(H)$. Then $H_P$ is seminormal and $\widetilde{H_P}$ is a DVM. By Theorem~\ref{theorem2}(1), it suffices to show that $\mathcal{A}(H_P)\subseteq\mathcal{A}(\widetilde{H_P})$. Since $H_P$ is primary and seminormal and $\widetilde{H_P}$ is a DVM, it follows from \cite[Lemma 3.3]{Ge-Ka-Re15} that $H_P\setminus H_P^{\times}=\widetilde{H_P}\setminus\widetilde{H_P}^{\times}$. Let $u\in\mathcal{A}(H_P)$. Clearly, $u\not\in\widetilde{H_P}^{\times}$. Let $x,y\in\widetilde{H_P}$ be such that $u=xy$. If $x,y\not\in\widetilde{H_P}^{\times}$, then $x,y\in H_P\setminus H_P^{\times}$, a contradiction. Consequently, $x\in\widetilde{H_P}^{\times}$ or $y\in\widetilde{H_P}^{\times}$.
\end{proof}

\begin{corollary}\label{corollary4} Let $H$ be a transfer Krull monoid such that $\widetilde{H}$ is Krull. If $H$ is weakly factorial or $H\subseteq\widetilde{H}$ is inert, then $H_P$ is transfer Krull for each $P\in\mathfrak{X}(H)$.
\end{corollary}

\begin{proof} Let $P\in\mathfrak{X}(H)$. If $H$ is weakly factorial, then it follows from Theorem~\ref{theorem3} that $H_P$ is transfer Krull. Now let $H\subseteq\widetilde{H}$ be inert. It is straightforward to show that $H_P\subseteq\widetilde{H}_P=\widetilde{H_P}$ is inert. Since $\widetilde{H}$ is Krull, it follows that $\widetilde{H_P}$ is a DVM. (This can be proved along the same lines as in the proof of Theorem~\ref{theorem3}.) We infer by Corollary~\ref{corollary2} that $H_P$ is transfer Krull.
\end{proof}

\section{(Generalized) Cohen-Kaplansky domains}

In this section we first gather some main properties of (generalized) Cohen-Kaplansky domains. By a {\it domain}, we mean a commutative integral domain with identity element. Let $R$ be a domain with quotient field $K$. We denote by $R^{\bullet}=R\setminus\{0\}$ the {\it multiplicative monoid} of nonzero elements, by $R^{\times}$ the {\it group of units} of $R$ and by $\overline{R}$ the {\it integral closure} of $R$. If A is a monoid theoretic property (e.g. factorial, half-factorial), then we say that $R$ satisfies A if $R^{\bullet}$ satisfies A. We say that $R$ is a {\it Cohen-Kaplansky} domain if one of the following equivalent statements holds (\cite[Theorem 4.3]{An-Mo92}).

\begin{itemize}
\item[(a)] $R$ is atomic and has only finitely many atoms up to associates.
\item[(b)] $\overline{R}$ is a semilocal principal ideal domain, $\overline R/(R:\overline{R})$ is finite, and $|\max(R)|=|\max(\overline{R})|$.
\item[(c)] $R$ is an at most one-dimensional semilocal noetherian domain with $R/M$ finite for each nonprincipal maximal ideal $M$ of $R$, $\overline{R}$ is a finitely generated $R$-module (equivalently $(R:\overline{R})\neq\{0\}$) and $|\max(R)|=|\max(\overline R)|$.
\item[(d)] $R$ is noetherian, $K^{\times}/R^{\times}$ (the group of divisibility) is finitely generated and for each $x\in\overline{R}$, there exists an $n\in\mathbb{N}$ such that $x^n\in R$ (that is $R\subseteq\overline{R}$ is a root extension).
\end{itemize}

For a local Cohen-Kaplansky domain $R$, the multiplicative monoid $R^{\bullet}$ is a finitely primary monoid of rank one. Another interesting fact is that if a domain $R$ is a Cohen-Kaplansky domain, then so are all the localizations and conversely if a domain $R$ is semilocal such that it is locally a Cohen-Kaplansky domain, then $R$ is also a Cohen-Kaplansky domain. Thus the study of Cohen-Kaplansky domains may be reduced to the local case for various purposes. Also note that $R$ is a Cohen-Kaplansky domain if and only if $R^{\bullet}$ is $s$-noetherian (see \cite[Page 137]{Ge-HK06}). We refer the reader to \cite{An-Bo21, An-Mo92} for further details about Cohen-Kaplansky domains.

Furthermore, we say that $R$ is a {\it generalized Cohen-Kaplansky} domain if one of the following equivalent statements holds (\cite[Corollary 5 and Theorem 6]{An-An-Za92}).

\begin{itemize}
\item[(a)] $R$ is atomic and has only finitely many atoms (up to associates) that are not prime elements.
\item[(b)] $\overline{R}$ is factorial, $R\subseteq\overline{R}$ is a root extension, $(R:\overline{R})$ is a principal ideal of $\overline{R}$ and $\overline{R}/(R:\overline{R})$ is finite.
\end{itemize}
If these equivalent conditions are satisfied, then $R$ is weakly factorial. In particular, if $R$ is a generalized Cohen-Kaplansky domain, then $R^{\bullet}$ is weakly factorial and $\widetilde{R^{\bullet}}$ is factorial. For non-trivial examples of generalized Cohen-Kaplansky domains we refer to \cite[Page 137]{Ge-HK06}.

\begin{proposition}\label{proposition7} Let $R$ be local Cohen-Kaplansky domain with maximal ideal $M$. Then the following conditions are equivalent.
\begin{enumerate}
\item $R$ is half-factorial.
\item $R$ is transfer Krull.
\item $R^{\bullet}\hookrightarrow\overline{R}^{\bullet}$ is a transfer homomorphism.
\end{enumerate}
If these equivalent conditions hold, then $M\overline{R}=\overline{R}\setminus\overline{R}^{\times}$.
\end{proposition}

\begin{proof} The equivalence is an immediate consequence of Corollary~\ref{corollary2}.

Now assume that the equivalent conditions hold. Then clearly, $\overline{R}=R\overline{R}^{\times}$. Let $t\in\overline{R}\setminus\overline{R}^{\times}$. Since $\overline{R}=R\overline{R}^{\times}$, there are $m\in R\setminus R^{\times}=M$ and $u\in\overline{R}^{\times}$ such that $t=mu\in M\overline{R}$. Therefore, $\overline{R}\setminus\overline{R}^{\times}\subseteq M\overline{R}$, and hence $M\overline{R}=\overline{R}\setminus\overline{R}^{\times}$.		
\end{proof}

Observe that the condition $M\overline{R}=\overline{R}\setminus\overline{R}^{\times}$ is not sufficient for $R$ to be half-factorial (\cite[Example 6.5]{An-Mo92}). Also note that the domain $R$ in \cite[Example 6.5]{An-Mo92} satisfies the property $\overline{R}=R\overline{R}^{\times}$ as well.

Now our aim is to provide a more general version of the above characterization. But for this we first recall that an integral domain $R$ is said to be a {\it QR-domain} if every overring of $R$ is a quotient overring (i.e., for each overring $T$ of $R$ there is some submonoid $S\subseteq R^{\bullet}$ such that $T=S^{-1}R$). Note that $R$ is a QR-domain if and only if $R$ is a Pr\"ufer domain for which the radical of every finitely generated ideal is the radical of a principal ideal \cite[Theorem 5]{Pe66}. In particular, every Dedekind domain with torsion class group is a QR-domain. Also, every B\'ezout domain (i.e., a domain in which every finitely generated ideal is principal) is a QR-domain.

\begin{proposition}\label{proposition8} Let $R$ be a Cohen-Kaplansky domain. The following statements are equivalent.
\begin{enumerate}
\item There is some overring $T$ of $R$ such that $T$ is a Krull domain and $R^{\bullet}\hookrightarrow T^{\bullet}$ is a transfer homomorphism.
\item $R^{\bullet}\hookrightarrow\overline{R}^{\bullet}$ is a transfer homomorphism.
\end{enumerate}
\end{proposition}

\begin{proof} (1) $\Rightarrow$ (2): Let $T$ be an overring of $R$ which is a Krull domain such that $R^{\bullet}\hookrightarrow T^{\bullet}$ is a transfer homomorphism. Note that $T^{\times}\cap R=R^{\times}$ and $T=RT^{\times}$. Since $R\subseteq T$ and $T$ is Krull, we have that $\overline{R}\subseteq\overline{T}=T$. Since $\overline{R}$ is a principal ideal domain, $\overline{R}$ is a QR-domain, and hence $T=S^{-1}\overline{R}$ for some submonoid $S\subseteq\overline{R}^{\bullet}$.

Next we show that $T^{\times}\cap\overline{R}=\overline{R}^{\times}$. ($\subseteq$): Let $x\in T^{\times}\cap\overline{R}$. There is some $k\in\mathbb{N}$ such that $x^k\in R$. We infer that $x^k\in T^{\times}\cap R=R^{\times}\subseteq\overline{R}^{\times}$. Since $x\in\overline{R}$, it follows that $x\in\overline{R}^{\times}$. ($\supseteq$): This is obvious.

Observe that $S\subseteq T^{\times}$. Consequently, $S\subseteq T^{\times}\cap\overline{R}=\overline{R}^{\times}$, and thus $T=S^{-1}\overline{R}=\overline{R}$. Therefore, $R^{\bullet}\hookrightarrow\overline{R}^{\bullet}$ is a transfer homomorphism.

(2) $\Rightarrow$ (1): This is clear, since $\overline{R}$ is both an overring of $R$ and a principal ideal domain, and hence $\overline{R}$ is a Krull domain.
\end{proof}

\begin{theorem}\label{theorem4} Let $R$ be a generalized Cohen-Kaplansky domain. The following statements are equivalent.
\begin{enumerate}
\item[(1)] $R$ is transfer Krull.
\item[(2)] $R_M$ is transfer Krull for each $M\in\max(R)$.
\item[(3)] $R_M$ is half-factorial for each $M\in\max(R)$.
\item[(4)] $R$ is half-factorial.
\item[(5)] $\mathcal{A}(\overline{R})=\{u\varepsilon\mid u\in\mathcal{A}(R),\varepsilon\in\overline{R}^{\times}\}$.
\item[(6)] $\mathcal{A}(R)\subseteq\mathcal{A}(\overline{R})$.
\end{enumerate}
In addition, if $R$ is seminormal, then these equivalent conditions are satisfied.
\end{theorem}

\begin{proof} This follows from Theorem~\ref{theorem3}, since $R$ is weakly factorial, $R\subseteq\overline{R}$ is a root extension and $\overline{R}$ is factorial.
\end{proof}

\section{Transfer Krull monoids and finitely generated monoids}

Let $H$ be a monoid. Then $H$ is called {\it affine} if $H$ is finitely generated and the quotient group of $H$ is torsion-free (equivalently, $H$ is isomorphic to a finitely generated additive submonoid of $\mathbb{Z}^d$ for some $d\in\mathbb{N}$). Also note that $H$ is reduced and affine if and only if $H$ is isomorphic to a finitely generated additive submonoid of $\mathbb{N}_0^d$ for some $d\in\mathbb{N}$. For a profound introduction to affine monoids we refer to \cite{Br-Gu09}. Clearly, affine monoids are finitely generated, and hence the root closure of an affine monoid is a Krull monoid by \cite[Propositions 2.7.4.2 and 2.7.11 and Theorems 2.6.5.1 and 2.7.13]{Ge-HK06}. First we want to point out that even finitely generated finitely primary monoids of rank one and exponent larger than one need not be transfer Krull.

\begin{example}\label{example1} Let $F$ be a DVM for which $F^{\times}$ is finite and cyclic and let $p\in F$ and $\alpha\in F^{\times}$ be such that $F\setminus F^{\times}=pF$ and $F^{\times}=\langle\{\alpha\}\rangle$. Set $n=|F^{\times}|$ and $H=[\{p\}\cup\{\alpha^k p^n\mid k\in [1,n-1]\}]$. Then $H$ is a reduced finitely generated finitely primary monoid of rank one and exponent $n$ and $\widetilde{H}=F$. Furthermore, if $n\geq 2$, then $H$ is not transfer Krull.
\end{example}

\begin{proof} Let $K$ be the quotient group of $F$. It is obvious that $K$ is the quotient group of $H$, $H\not=K$ and $H$ is finitely generated. Clearly, $H\subseteq F=[\{p,\alpha\}]$, and hence $\widetilde{H}\subseteq\widehat{H}\subseteq\widehat{F}=F$. Since $x^n\in H$ for each $x\in F$, we infer that $\widetilde{H}=\widehat{H}=F$. Observe that $H^{\times}=\widetilde{H}^{\times}\cap H=F^{\times}\cap H=\{1\}$, and thus $H$ is reduced. Let $a,b\in H\setminus H^{\times}$. Then $a,b\in F\setminus F^{\times}$, and since $F$ is primary, there are some $k\in\mathbb{N}$ and some $c\in F$ such that $b^k=ca$. We infer that $b^{kn}=c^na^n$, and since $c^n\in H$, it follows that $a\mid_H b^{kn}$. Therefore, $H$ is primary. Note that $\mathfrak{X}(\widehat{H})=\{p\widehat{H}\}$ and $p^n\widehat{H}\subseteq (H:\widehat{H})$. Together with the fact that $\widehat{H}$ is a DVM, this implies that $H$ is a finitely primary monoid of rank one and exponent $n$. Now let $n\geq 2$ and assume that $H$ is transfer Krull. It follows from Corollary~\ref{corollary2} that $H\subseteq\widetilde{H}$ is inert. Since $p\alpha p^{n-1}\in H$, there is some $k\in [0,n-1]$ such that $\alpha^k p,\alpha^{1-k}p^{n-1}\in H$. Since $\alpha^k p\in H$, we have that $k=0$, and hence $\alpha p^{n-1}\in H$, a contradiction.
\end{proof}

Next we show that even if $H$ is a factorial monoid with finitely many prime elements, then $H$ need not satisfy property (U).

\begin{example}\label{example2} Let $H=\mathbb{N}_0^2$ and $T=\{(x,y)\in\mathbb{Z}^2\mid x+y\geq 0\}$. Then $H$ is factorial with precisely two prime elements, $T$ is a DVM and an overmonoid of $H$, $H\hookrightarrow T$ is a transfer homomorphism and $H$ does not satisfy property (U).
\end{example}

\begin{proof} Obviously, $H$ is a factorial monoid and $\{(0,1),(1,0)\}$ is the set of prime elements of $H$. It is also easy to see that $T$ is an overmonoid of $H$. Since $H$ is not a valuation monoid, it follows from Proposition~\ref{proposition6} that $H$ does not satisfy property (U). Note that $T^{\times}=\{(k,-k)\mid k\in\mathbb{Z}\}$ and $T=\{(n,0)+(k,-k)\mid n\in\mathbb{N}_0,k\in\mathbb{Z}\}$. This implies that $T$ is a DVM (since $T$ is atomic and has precisely one atom up to associates) and $T^{\times}\cap H=H^{\times}$. By Lemma~\ref{lemma1}(2), it remains to show that $H\subseteq T$ is inert. Let $a,b\in T$ be such that $a+b\in H$. There are some $x,y,z,w\in\mathbb{Z}$ such that $a=(x,y)$, $b=(z,w)$ and $x+y,z+w\geq 0$. Since $a+b\in H$, it follows that $x+z,y+w\geq 0$. Set $k=\min\{y,z\}$ and $\varepsilon=(k,-k)$. Then $\varepsilon\in T^{\times}$ and $a+\varepsilon,b-\varepsilon\in H$.
\end{proof}

Let $H$ be a monoid. Then $H$ is called {\it weakly Krull} if $H=\bigcap_{\mathfrak{p}\in\mathfrak{X}(H)} H_{\mathfrak{p}}$ and $\{\mathfrak{p}\in\mathfrak{X}(H)\mid a\in\mathfrak{p}\}$ is finite for each $a$ in $H$. Note that $H$ is a Krull monoid if and only if it is weakly Krull and $H_\mathfrak{p}$ is a DVM for each $\mathfrak{p}\in\mathfrak{X}(H)$. A Mori monoid $H$ which is not a group is weakly Krull if and only if $t$-$\max(H)=\mathfrak{X}(H)$ by \cite[Theorem 24.5]{HK98}. Moreover, $H$ is weakly factorial if and only if $H$ is weakly Krull and every $t$-invertible $t$-ideal of $H$ is principal. Weakly Krull monoids were studied by Halter-Koch (\cite[Chapter 22 and Chapter 24.5]{HK98}). Clearly, every primary monoid is a weakly Krull monoid. Observe that even a finitely generated monoid whose root closure is factorial need not be weakly Krull (\cite[Example 2]{Ch-HK-K02}). For more information on weakly Krull monoids (and weakly Krull domains) we refer to \cite{An-Mo-Za92, HK95, HK98}.

\begin{proposition}\label{proposition9} Let $F$ be the free abelian monoid with basis $\{a,b\}$ and quotient group $K$ and let $H$ be the submonoid of $F$ generated by $\{a,ab,a^2b^5\}$.
\begin{enumerate}
\item[(1)] $H$ is a reduced affine monoid with quotient group $K=\{a^rb^s\mid r,s\in\mathbb{Z}\}$, $\widetilde{H}\subseteq F$ and $\mathcal{A}(H)=\{a,ab,a^2b^5\}$.
\item[(2)] $\widetilde{H}=\{a^rb^s\mid r,s\in\mathbb{N}_0,5r\geq 2s\}=[\{a,ab,ab^2,a^2b^5\}]$ and $\mathcal{A}(H)\subseteq\mathcal{A}(\widetilde{H})$.
\item[(3)] $\mathfrak{X}(H)=\{aH\cup abH,abH\cup a^2b^5H\}$ and $H$ is weakly Krull.
\item[(4)] There is some overmonoid $B$ of $H$ such that $B$ is Krull, $B^{\times}\cap H=H^{\times}$ and $B=HB^{\times}$.
\item[(5)] There is no overmonoid $T$ of $\widetilde{H}$ such that $\mathcal{A}(H)\subseteq\mathcal{A}(T)$ and $T=HT^{\times}$. In particular, $H$ is not transfer Krull.
\end{enumerate}
\end{proposition}

\begin{proof} (1) This is clear.

\medskip
(2) First let $x\in\widetilde{H}$. Then $x=a^rb^s$ for some $r,s\in\mathbb{N}_0$ and $x^k\in\mathbb{N}$ for some $k\in\mathbb{N}$. There are some $\alpha,\beta,\gamma\in\mathbb{N}_0$ such that $kr=\alpha+\beta+2\gamma$ and $ks=\beta+5\gamma$. Then $(5r-2s)k=5\alpha+3\beta\geq 0$, and hence $5r\geq 2s$.

Let $r,s\in\mathbb{N}_0$ be such that $5r\geq 2s$. Then $s=5q+m$ for some $q\in\mathbb{N}_0$ and $m\in [0,4]$. Moreover, there are some $p\in [0,2]$ and $n\in\{0,1\}$ with $m=2p+n$. We obtain that $r\geq 2q+p+n$ and $a^rb^s=a^{r-2q-p-n}(ab)^n(ab^2)^p(a^2b^5)^q\in [\{a,ab,ab^2,a^2b^5\}]$.

Finally, we have that $(ab^2)^3=ab(a^2b^5)\in H$, and thus $[\{a,ab,ab^2,a^2b^5\}]\subseteq\widetilde{H}$. It is now straightforward to show that $\mathcal{A}(H)\subseteq\{a,ab,ab^2,a^2b^5\}=\mathcal{A}(\widetilde{H})$.

\medskip
(3) Set $P=aH\cup abH$ and $Q=abH\cup a^2b^5H$. Note that if $N$ is a prime $s$-ideal of $H$, then $N=\bigcup_{u\in\mathcal{A}(H)\cap N} uH$. Since $(ab)^5=(a^2b^5)a^3$, we obtain that $aH$, $abH$, $a^2b^5H$ and $aH\cup a^2b^5H$ are not prime $s$-ideals of $H$. Next we show that $P$ and $Q$ are prime. (Then we conclude by the aforementioned facts that $\mathfrak{X}(H)=\{P,Q\}$.)

Let $x,y\in H\setminus P$. Then $xy=(a^2b^5)^k$ for some $k\in\mathbb{N}_0$. Let $r,s,t\in\mathbb{N}_0$ be such that $xy=a^r(ab)^s(a^2b^5)^t$. Consequently, $2k=r+s+2t$ and $5k=s+5t$, and thus $5r+5s+10t=10k=2s+10t$. We infer that $5r+3s=0$ and $r=s=0$. Therefore, $xy\in H\setminus P$.

Let $x,y\in H\setminus Q$. Then $xy=a^k$ for some $k\in\mathbb{N}_0$. Let $r,s,t\in\mathbb{N}_0$ be such that $xy=a^r(ab)^s(a^2b^5)^t$. Then $k=r+s+2t$ and $0=s+5t$. This implies that $s=t=0$. Consequently, $xy\in H\setminus Q$.

Finally, we show that $H$ is weakly Krull. Since $H$ has finitely many prime $s$-ideals, it remains to show that $\bigcap_{N\in\mathfrak{X}(H)} H_N=H$. Note that $H_P=\{a^r(ab)^s(a^2b^5)^t\mid r,s\in\mathbb{N}_0,t\in\mathbb{Z}\}$ and $H_Q=\{a^r(ab)^s(a^2b^5)^t\mid r\in\mathbb{Z},s,t\in\mathbb{N}_0\}$. Let $x\in H_P\cap H_Q$. Then $x=a^r(ab)^s(a^2b^5)^t=a^{r^{\prime}}(ab)^{s^{\prime}}(a^2b^5)^{t^{\prime}}$ for some $r,s,s^{\prime},t^{\prime}\in\mathbb{N}_0$ and $r^{\prime},t\in\mathbb{Z}$. It follows that $r+s+2t=r^{\prime}+s^{\prime}+2t^{\prime}$ and $s+5t=s^{\prime}+5t^{\prime}$, and hence $r^{\prime}+3t=r+3t^{\prime}\geq 0$. Consequently, $r^{\prime}\geq 0$ or $t\geq 0$, and thus $x\in H$.

\medskip
(4) Set $B=\{a^rb^s\mid r\in\mathbb{N}_0,s\in\mathbb{Z}\}$. Clearly, $B$ is an overmonoid of $H$, $B$ is a DVM and $B^{\times}=\{b^s\mid s\in\mathbb{Z}\}$. Therefore, $B^{\times}\cap H=H^{\times}$ and $B=HB^{\times}$.

\medskip
(5) Assume that there is some overmonoid $T$ of $\widetilde{H}$ such that $\mathcal{A}(H)\subseteq\mathcal{A}(T)$ and $T=HT^{\times}$. We infer by Proposition~\ref{proposition5} that $T^{\times}\cap H=H^{\times}$. Note that $ab^2\in T\setminus T^{\times}$ (since $T^{\times}\cap\widetilde{H}=\{1\}$) and $b\not\in T^{\times}$ (for if $b\in T^{\times}$, then $a^2\simeq_T a^2b^5\in\mathcal{A}(T)$, a contradiction). Since $ab^2\in HT^{\times}$, there are some $r,s,t\in\mathbb{N}_0$ such that $a^{1-r}b^2(ab)^{-s}(a^2b^5)^{-t}\in T^{\times}$.

First we assume that $r>0$. Then $b^2\in T$, and hence $a^2b^5=a(ab)(b^2)^2\not\in\mathcal{A}(T)$, a contradiction. Therefore, $r=0$.

Next we assume that $t>0$. Then $a^{-1}b^{-3}(ab)^{-s}(a^2b^5)^{1-t}\in T^{\times}$, and thus $a^{-1}b^{-3}\in T$. This implies that $b^{-1}=ab^2a^{-1}b^{-3}\in T$. Since $ab^2b^{-1}=ab\in\mathcal{A}(T)$, we have that $b^{-1}\in T^{\times}$, and hence $b\in T^{\times}$, a contradiction. We infer that $t=0$.

Since $b,ab^2\not\in T^{\times}$, it follows that $s>1$. Consequently, $(ab)^s(ab^2)^{-1}\in T^{\times}\cap H=\{1\}$, and thus $ab^2\in H$, a contradiction.
\end{proof}

Note that by Proposition~\ref{proposition9}, we have that conditions b and c in Theorem~\ref{theorem2} are no longer equivalent if $H$ is a reduced affine weakly Krull monoid.

\begin{proposition}\label{proposition10} Let $F$ be the free abelian monoid with basis $\{a,b\}$ and quotient group $K$ and let $H$ be the submonoid of $F$ generated by $\{a,ab^3,ab^5\}$.
\begin{enumerate}
\item[(1)] $H$ is a reduced affine monoid with quotient group $K=\{a^rb^s\mid r,s\in\mathbb{Z}\}$, $\widetilde{H}\subseteq F$ and $\mathcal{A}(H)=\{a,ab^3,ab^5\}$.
\item[(2)] $\widetilde{H}=\{a^rb^s\mid r,s\in\mathbb{N}_0,5r\geq s\}=[\{a,ab,ab^2,ab^3,ab^4,ab^5\}]$ and $\mathcal{A}(H)\subseteq\mathcal{A}(\widetilde{H})$.
\item[(3)] $\mathfrak{X}(H)=\{aH\cup ab^3H,ab^3H\cup ab^5H\}$, $H$ is half-factorial and weakly Krull and $\widetilde{H}$ is half-factorial.
\item[(4)] $\mathcal{A}(\widetilde{H})\not=\{u\varepsilon\mid u\in\mathcal{A}(H),\varepsilon\in\widetilde{H}^{\times}\}$.
\item[(5)] There is some $P\in\mathfrak{X}(H)$ such that $H_P$ is not transfer Krull.
\end{enumerate}
\end{proposition}

\begin{proof} (1) This is straightforward to prove.

\medskip
(2) Let $x\in\widetilde{H}$. There are some $r,s\in\mathbb{N}_0$ and some $k\in\mathbb{N}$ such that $x=a^rb^s$ and $x^k\in H$. Consequently, there are some $\alpha,\beta,\gamma\in\mathbb{N}_0$ such that $kr=\alpha+\beta+\gamma$ and $ks=3\beta+5\gamma$. We infer that $(5r-s)k=5\alpha+2\beta\geq 0$. This implies that $5r\geq s$.

Now let $r,s\in\mathbb{N}_0$ be such that $5r\geq s$. There are some $q\in\mathbb{N}_0$ and $m\in [0,4]$ such that $s=5q+m$. Set $n=\lceil\frac{m}{5}\rceil$. It follows that $r\geq q+n$ and $a^rb^s=a^{r-q-n}(ab^m)^n(ab^5)^q\in [\{a,ab,ab^2,ab^3,ab^4,ab^5\}]$. If $g\in [0,5]$, then $(ab^g)^5\in H$ and thus $[\{a,ab,ab^2,ab^3,ab^4,ab^5\}]\subseteq\widetilde{H}$.

It is now easy to see that $\mathcal{A}(H)\subseteq\{a,ab,ab^2,ab^3,ab^4,ab^5\}=\mathcal{A}(\widetilde{H})$.

\medskip
(3) It is obvious that $H$ and $\widetilde{H}$ are half-factorial (e.g. see \cite[Lemma 2]{Ga-Oj-Sa13}). Set $P=aH\cup ab^3H$ and $Q=ab^3H\cup ab^5H$. Since $a^2(ab^5)^3=(ab^3)^5$, we have that $aH$, $abH$, $ab^5H$ and $aH\cup ab^5H$ are not prime $s$-ideals of $H$. Moreover, it is easy to show that $P$ and $Q$ are prime $s$-ideals of $H$. Therefore, $\mathfrak{X}(H)=\{P,Q\}$. Observe that $H_P=\{a^r(ab^3)^s(ab^5)^t\mid r,s\in\mathbb{N}_0,t\in\mathbb{Z}\}$ and $H_Q=\{a^r(ab^3)^s(ab^5)^t\mid r\in\mathbb{Z},s,t\in\mathbb{N}_0\}$. It remains to show that $\bigcap_{N\in\mathfrak{X}(H)} H_N=H$. Let $x\in H_P\cap H_Q$. Then $x=a^r(ab^3)^s(ab^5)^t=a^{r^{\prime}}(ab^3)^{s^{\prime}}(ab^5)^{t^{\prime}}$ for some $r,s,s^{\prime},t^{\prime}\in\mathbb{N}$ and $r^{\prime},t\in\mathbb{Z}$. This implies that $r+s+t=r^{\prime}+s^{\prime}+t^{\prime}$ and $3s+5t=3s^{\prime}+5t^{\prime}$. Therefore, $3r-2t=3r^{\prime}-2t^{\prime}$, and hence $3r^{\prime}+2t=3r+2t^{\prime}\geq 0$. We infer that $r^{\prime}\geq 0$ or $t\geq 0$, and thus $x\in H$.

\medskip
(4) This is clear, since $\widetilde{H}^{\times}=\{1\}$ and $ab\in\mathcal{A}(\widetilde{H})\setminus\mathcal{A}(H)$.

\medskip
(5) Set $P=aH\cup ab^3H$. Then $P\in\mathfrak{X}(H)$ by (3) and $H_P=\{a^r(ab^3)^s(ab^5)^t\mid r,s\in\mathbb{N}_0,t\in\mathbb{Z}\}$. Observe that $H_P^{\times}=\{(ab^5)^k\mid k\in\mathbb{Z}\}$ and $\mathcal{A}(H_P)=\{a(ab^5)^k,ab^3(ab^5)^k\mid k\in\mathbb{Z}\}$. Since $a^2(ab^5)^3=(ab^3)^5$, we have that $2,5\in\mathsf{L}_{H_P}(a^5b^{15})$, and hence $H_P$ is not half-factorial. Therefore, $H_P$ is not transfer Krull by Corollary~\ref{corollary2} (since $\widetilde{H_P}$ is a DVM, see the proof of Theorem~\ref{theorem3}).
\end{proof}

We obtain by Proposition~\ref{proposition10} that a half-factorial monoid whose root closure is half-factorial and Krull need not satisfy the equivalent conditions in Theorem~\ref{theorem2}(2). It is also pointed out by this result that if $H$ is a half-factorial weakly Krull monoid whose root closure is Krull, then the localization $H_P$ (where $P\in\mathfrak{X}(H)$) need not be transfer Krull (in contrast to Corollary~\ref{corollary4}).

\begin{proposition}\label{proposition11} Let $F$ be the free abelian monoid with basis $\{a,b,c,d\}$ and quotient group $K$ and let $H$ be the submonoid of $F$ generated by $\{ab,ac,ad,abc,bcd\}$.
\begin{enumerate}
\item[(1)] $H$ is a reduced affine monoid with quotient group $K=\{a^rb^sc^td^u\mid r,s,t,u\in\mathbb{Z}\}$, $\widetilde{H}\subseteq F$ and $\mathcal{A}(H)=\{ab,ac,ad,abc,bcd\}$.
\item[(2)] $\widetilde{H}=\{a^rb^sc^td^u\mid r,s,t,u\in\mathbb{N}_0,r\leq s+t+u,s\leq r+\min\{t,u\},t\leq r+\min\{s,u\},u\leq r+\min\{s,t\},s+u\leq r+2t,t+u\leq r+2s\}=[\{ab,ac,ad,abc,bcd,abcd\}]$.
\item[(3)] $\mathfrak{X}(H)=\{abH\cup adH,acH\cup adH,adH\cup bcdH,abH\cup abcH,acH\cup abcH,abcH\cup bcdH\}$ and $H$ is weakly Krull.
\item[(4)] There is some overmonoid $T$ of $H$ such that $T$ is Krull and $\mathcal{A}(T)=\{u\varepsilon\mid u\in\mathcal{A}(H),\varepsilon\in T^{\times}\}$.
\item[(5)] $H$ is not transfer Krull.
\end{enumerate}
\end{proposition}

\begin{proof} Claim: If $(ab)^{\alpha}(ac)^{\beta}(ad)^{\gamma}(abc)^{\delta}(bcd)^{\varepsilon}=(ab)^{\alpha^{\prime}}(ac)^{\beta^{\prime}}(ad)^{\gamma^{\prime}}(abc)^{\delta^{\prime}}(bcd)^{\varepsilon^{\prime}}$ for $\alpha,\beta,\gamma,\delta,\varepsilon\in\mathbb{Z}$ and $\alpha^{\prime},\beta^{\prime},\gamma^{\prime},\delta^{\prime},\varepsilon^{\prime}\in\mathbb{Z}$, then there is some $k\in\mathbb{Z}$ such that $(\alpha,\beta,\gamma,\delta,\varepsilon)=(\alpha^{\prime},\beta^{\prime},\gamma^{\prime},\delta^{\prime},\varepsilon^{\prime})+k(-2,-2,1,3,-1)$.

\medskip
(1) and the claim are straightforward to prove.

\medskip
(2) Set $A=\{a^rb^sc^td^u\mid r,s,t,u\in\mathbb{N}_0,r\leq s+t+u,s\leq r+\min\{t,u\},t\leq r+\min\{s,u\},u\leq r+\min\{s,t\},s+u\leq r+2t,t+u\leq r+2s\}$ and set $B=[\{ab,ac,ad,abc,bcd,abcd\}]$.

\smallskip
First we prove that $\widetilde{H}\subseteq A$. Let $x\in\widetilde{H}$. Clearly, there are some $r,s,t,u\in\mathbb{N}_0$ and $k\in\mathbb{N}$ such that $x=a^rb^sc^td^u$ and $x^k\in H$. Consequently, there are some $\alpha,\beta,\gamma,\delta,\varepsilon\in\mathbb{N}_0$ such that $kr=\alpha+\beta+\gamma+\delta$, $ks=\alpha+\delta+\varepsilon$, $kt=\beta+\delta+\varepsilon$ and $ku=\gamma+\varepsilon$. Note that $(s+t+u-r)k=\delta+3\varepsilon\geq 0$, $(r+u-s)k=\beta+2\gamma\geq 0$, $(r+t-s)k=2\beta+\gamma+\delta\geq 0$, $(r+s-t)k=2\alpha+\gamma+\delta\geq 0$, $(r+u-t)k=\alpha+2\gamma\geq 0$, $(r+s-u)k=2\alpha+\beta+2\delta\geq 0$, $(r+t-u)k=\alpha+2\beta+2\delta\geq 0$, $(r+2t-s-u)k=3\beta+2\delta\geq 0$ and $(r+2s-t-u)k=3\alpha+2\delta\geq 0$. This implies that $x\in A$.

\smallskip
Next we prove that $A\subseteq B$. It suffices to show by induction that for all $r,s,t,u\in\mathbb{N}_0$ with $a^rb^sc^td^u\in A$, it follows that $a^rb^sc^td^u\in B$. Let $r,s,t,u\in\mathbb{N}_0$ be such that $a^rb^sc^td^u\in A$.

Case 1: $r=0$. Observe that $s=t=u$, and hence $a^rb^sc^td^u=(bcd)^s\in B$.

Case 2: $s=0$. We have that $r=t+u$, and thus $a^rb^sc^td^u=(ac)^t(ad)^u\in B$.

Case 3: $t=0$. It follows that $r=s+u$. Consequently, $a^rb^sc^td^u=(ab)^s(ad)^u\in B$.

Case 4: $u=0$. Note that $s,t\leq r$ and $r\leq s+t$. We infer that $a^rb^sc^td^u=(ab)^{r-t}(ac)^{r-s}(abc)^{s+t-r}\in B$.

Case 5: $r=s+t+u$. Then $a^rb^sc^td^u=(ab)^s(ac)^t(ad)^u\in B$.

Case 6: $r+1=s+t+u$. It follows that $s,t>0$. Consequently, $a^rb^sc^td^u=(ab)^{s-1}(ac)^{t-1}(ad)^uabc\in B$.

Case 7: $s=r+t$. Clearly, $t=u$. Therefore, $a^rb^sc^td^u=(ab)^r(bcd)^t\in B$.

Case 8: $s=r+u$. Observe that $u\leq s\leq t$. This implies that $a^rb^sc^td^u=(ab)^{s-t}(abc)^{t-u}(bcd)^u\in B$.

Case 9: $t=r+s$. Obviously, $s=u$, and hence $a^rb^sc^td^u=(ac)^r(bcd)^s\in B$.

Case 10: $t=r+u$. We have that $u\leq s\leq t$. It follows that $a^rb^sc^td^u=(ac)^{t-s}(abc)^{s-u}(bcd)^u\in B$.

Case 11: $u=r+\min\{s,t\}$. We infer that $s=t$, and thus $a^rb^sc^td^u=(ad)^r(bcd)^s\in B$.

Case 12: $r+2t=s+u$. Observe that $t\leq s,u$. Therefore, $a^rb^sc^td^u=(ab)^{s-t}(ad)^{u-t}(bcd)^t\in B$.

Case 13: $r+2s=t+u$. Clearly, $s\leq t,u$. This implies that $a^rb^sc^td^u=(ac)^{t-s}(ad)^{u-s}(bcd)^s\in B$.

Now assume that none of the above cases applies. Set $r^{\prime}=r-1$, $s^{\prime}=s-1$, $t^{\prime}=t-1$ and $u^{\prime}=u-1$. It is straightforward to show that $a^{r^{\prime}}b^{s^{\prime}}c^{t^{\prime}}d^{u^{\prime}}\in A$. Since $r^{\prime}<r$, we infer by the induction hypothesis that $a^{r^{\prime}}b^{s^{\prime}}c^{t^{\prime}}d^{u^{\prime}}\in B$. Consequently, $a^rb^sc^td^u=abcda^{r^{\prime}}b^{s^{\prime}}c^{t^{\prime}}d^{u^{\prime}}\in B$.

\smallskip
Finally, we show that $B\subseteq\widetilde{H}$. Since $(abcd)^2=(ad)(abc)(bcd)\in H$ and $H\subseteq\widetilde{H}$, it is clear that $B\subseteq\widetilde{H}$.

\medskip
(3) It is an easy consequence of the claim that $\mathfrak{X}(H)=\{abH\cup adH,acH\cup adH,adH\cup bcdH,abH\cup abcH,acH\cup abcH,abcH\cup bcdH\}$. It remains to show that $\bigcap_{P\in\mathfrak{X}(H)} H_P=H$. Let $x\in\bigcap_{P\in\mathfrak{X}(H)} H_P$. There are some $(\alpha_i)_{i=1}^6,(\beta_i)_{i=1}^6,(\gamma_i)_{i=1}^6,(\delta_i)_{i=1}^6,(\varepsilon_i)_{i=1}^6\in\mathbb{Z}^6$ such that $x=(ab)^{\alpha_j}(ac)^{\beta_j}(ad)^{\gamma_j}(abc)^{\delta_j}(bcd)^{\varepsilon_j}$ for each $j\in [1,6]$ and  $\alpha_1,\gamma_1,\beta_2,\gamma_2,\gamma_3,\varepsilon_3,\alpha_4,\delta_4,\beta_5,\delta_5,\delta_6,\varepsilon_6\in\mathbb{N}_0$.

We infer by the claim that $x\in H$ if and only if there is some $k\in\mathbb{Z}$ such that $\alpha_1-2k,\beta_1-2k,\gamma_1+k,\delta_1+3k,\varepsilon_1-k\geq 0$ if and only if there is some $k\in\mathbb{Z}$ such that $\min\{\lfloor\frac{\alpha_1}{2}\rfloor,\lfloor\frac{\beta_1}{2}\rfloor,\varepsilon_1\}\geq k\geq\max\{-\gamma_1,-\frac{\delta_1}{3}\}$ if and only if $\min\{\lfloor\frac{\alpha_1}{2}\rfloor,\lfloor\frac{\beta_1}{2}\rfloor,\varepsilon_1\}\geq\max\{-\gamma_1,-\frac{\delta_1}{3}\}$. Therefore, it remains to prove the six inequalities $\lfloor\frac{\alpha_1}{2}\rfloor+\gamma_1\geq 0$, $\lfloor\frac{\beta_1}{2}\rfloor+\gamma_1\geq 0$, $\varepsilon_1+\gamma_1\geq 0$, $\lfloor\frac{\alpha_1}{2}\rfloor+\frac{\delta_1}{3}\geq 0$, $\lfloor\frac{\beta_1}{2}\rfloor+\frac{\delta_1}{3}\geq 0$ and $\varepsilon_1+\frac{\delta_1}{3}\geq 0$.

By the claim, there is some sequence $(k_i)_{i=1}^5\in\mathbb{Z}^5$ such that $(\alpha_1,\beta_1,\gamma_1,\delta_1,\varepsilon_1)=(\alpha_j,\beta_j,\gamma_j,\delta_j,\varepsilon_j)+k_{j-1}(-2,-2,1,3,-1)$ for each $j\in [2,6]$.

Since $\alpha_1,\gamma_1\geq 0$, we have that $\lfloor\frac{\alpha_1}{2}\rfloor+\gamma_1\geq 0$. By using $\beta_2,\gamma_2\geq 0$, we infer that $\lfloor\frac{\beta_1}{2}\rfloor+\gamma_1=\lfloor\frac{\beta_2-2k_1}{2}\rfloor+\gamma_2+k_1=\lfloor\frac{\beta_2}{2}\rfloor+\gamma_2\geq 0$. Since $\gamma_3,\varepsilon_3\geq 0$, it follows that $\varepsilon_1+\gamma_1=\varepsilon_3-k_2+\gamma_3+k_2=\varepsilon_3+\gamma_3\geq 0$. By using $\alpha_4,\delta_4\geq 0$, we have that $\lfloor\frac{\alpha_1}{2}\rfloor+\frac{\delta_1}{3}=\lfloor\frac{\alpha_4-2k_3}{2}\rfloor+\frac{\delta_4+3k_3}{3}=\lfloor\frac{\alpha_4}{2}\rfloor+\frac{\delta_4}{3}\geq 0$. Since $\beta_5,\delta_5\geq 0$, we infer that $\lfloor\frac{\beta_1}{2}\rfloor+\frac{\delta_1}{3}=\lfloor\frac{\beta_5-2k_4}{2}\rfloor+\frac{\delta_5+3k_4}{3}=\lfloor\frac{\beta_5}{2}\rfloor+\frac{\delta_5}{3}\geq 0$. Finally, by using $\delta_6,\varepsilon_6\geq 0$, it follows that $\varepsilon_1+\frac{\delta_1}{3}=\varepsilon_6-k_5+\frac{\delta_6+3k_5}{3}=\varepsilon_6+\frac{\delta_6}{3}\geq 0$.

\medskip
(4) Set $T=[\{a,ac,ad,cd,b,b^{-1}\}]$. Clearly, $T$ is a finitely generated overmonoid of $H$ and $T^{\times}=\{b^k\mid k\in\mathbb{Z}\}$. Observe that $T=\{a^rb^sc^td^u\mid r,t,u\in\mathbb{N}_0,s\in\mathbb{Z},t\leq r+u,u\leq r+t\}$. We infer that $T$ is root closed (and hence it is Krull by \cite[Theorem 2.7.14]{Ge-HK06}) and $\mathcal{A}(T)=\{ab^k,ab^kc,ab^kd,b^kcd\mid k\in\mathbb{Z}\}$. It is now straightforward to prove that $\mathcal{A}(T)=\{u\varepsilon\mid u\in\mathcal{A}(H),\varepsilon\in T^{\times}\}$.

\medskip
(5) Assume that $H$ is transfer Krull. By Proposition~\ref{proposition1} and Lemma~\ref{lemma1}(2) there is some overmonoid $B$ of $H$ such that $B$ is Krull, $B^{\times}\cap H=H^{\times}$ and $H\subseteq B$ is inert. (In particular, $\mathcal{A}(H)\subseteq\mathcal{A}(B)$.) Note that $(abcd)^2=(ad)(abc)(bcd)\in H$ and $abcd\in\widetilde{H}\subseteq B$. Consequently, there is some $\varepsilon\in B^{\times}$ such that $abcd\varepsilon,abcd\varepsilon^{-1}\in H$. It follows by the claim that $\{x\in H\mid x\mid_H (abcd)^2\}=\{1,ad,abc,bcd,a^2bcd,abcd^2,ab^2c^2d,a^2b^2c^2d^2\}$.

We infer that $\varepsilon\in\{a^{-1}b^{-1}c^{-1}d^{-1},b^{-1}c^{-1},d^{-1},a^{-1},a,d,bc,abcd\}$. Since $B^{\times}\cap\widetilde{H}=\widetilde{H}^{\times}=\{1\}$, this implies that $\{a,d,bc\}\cap B^{\times}\not=\emptyset$. If $bc\in B^{\times}$, then $a,d\in B$, and hence $a\in B^{\times}$ or $d\in B^{\times}$ (since $ad\in\mathcal{A}(B)$).

Case 1: $a\in B^{\times}$. Since $ab,ac\in B\setminus B^{\times}$, we have that $b,c\in B\setminus B^{\times}$. Moreover, $bc\simeq_B abc\in\mathcal{A}(B)$, and thus $bc\in\mathcal{A}(B)$, a contradiction.

Case 2: $d\in B^{\times}$. Since $ad,bcd\in B\setminus B^{\times}$, it follows that $a,bc\in B\setminus B^{\times}$. Furthermore, $abc\in\mathcal{A}(B)$, a contradiction.
\end{proof}

Note that Proposition~\ref{proposition11} shows that conditions a and b in Theorem~\ref{theorem2}(1) are no longer equivalent for reduced affine monoids.

\begin{example}\label{example3} Let $F$ be the free abelian monoid with basis $\{a,b\}$ and quotient group $K$, let $H$ be the submonoid of $F$ generated by $\{a^2,b^2,ab,a^2b,ab^2\}$, let $T_1$ be the submonoid of $K$ generated by $\{a,ab,b^2,b^{-2}\}$ and let $T_2$ be the submonoid of $K$ generated by $\{b,ab,a^2,a^{-2}\}$.
\begin{enumerate}
\item[(1)] $H$ is a reduced affine monoid with quotient group $K=\{a^rb^s\mid r,s\in\mathbb{Z}\}$, $\widetilde{H}=F$ and $\mathcal{A}(H)=\{a^2,b^2,ab,a^2b,ab^2\}$.
\item[(2)] $\mathcal{A}(H)\nsubseteq\mathcal{A}(\widetilde{H})$ and, in particular, $H$ is not transfer Krull.
\item[(3)] $T_1$ and $T_2$ are seminormal half-factorial overmonoids of $H$ and $H=T_1\cap T_2$.
\item[(4)] $\mathfrak{X}(H)=\{a^2H\cup abH\cup a^2bH\cup ab^2H,b^2H\cup abH\cup a^2bH\cup ab^2H\}$ and $H$ is seminormal and weakly Krull.
\end{enumerate}
\end{example}

\begin{proof} (1) This is obvious.

\medskip
(2) Note that $\mathcal{A}(H)=\{a^2,b^2,ab,a^2b,ab^2\}\nsubseteq\{a,b\}=\mathcal{A}(\widetilde{H})$. The remaining statement follows from Lemma~\ref{lemma2}(3).

\medskip
(3) Clearly, $T_1$ and $T_2$ are overmonoids of $H$, $\widetilde{T_1}=[\{a,b,b^{-1}\}]$ is a DVM and $\widetilde{T_2}=[\{a,b,a^{-1}\}]$ is a DVM. Moreover, $T_1^{\times}=\{b^{2r}\mid r\in\mathbb{Z}\}$, $T_2^{\times}=\{a^{2r}\mid r\in\mathbb{Z}\}$, $T_1\setminus T_1^{\times}=\widetilde{T_1}\setminus\widetilde{T_1}^{\times}=\{a^rb^s\mid r\in\mathbb{N},s\in\mathbb{Z}\}$ and $T_2\setminus T_2^{\times}=\widetilde{T_2}\setminus\widetilde{T_2}^{\times}=\{a^rb^s\mid r\in\mathbb{Z},s\in\mathbb{N}\}$. Therefore, $T_1$ and $T_2$ are seminormal. Furthermore, $\mathcal{A}(T_1)=\{ab^s\mid s\in\mathbb{Z}\}=\mathcal{A}(\widetilde{T_1})$ and $\mathcal{A}(T_2)=\{a^rb\mid r\in\mathbb{Z}\}=\mathcal{A}(\widetilde{T_2})$. Consequently, $T_1$ and $T_2$ are half-factorial.

Finally, note that $H=\{1\}\cup\{a^{2r}\mid r\in\mathbb{N}\}\cup\{b^{2s}\mid s\in\mathbb{N}\}\cup\{a^rb^s\mid r,s\in\mathbb{N}\}=(\{b^{2s}\mid s\in\mathbb{Z}\}\cup\{a^rb^s\mid r\in\mathbb{N},s\in\mathbb{Z}\})\cap (\{a^{2s}\mid s\in\mathbb{Z}\}\cup\{a^rb^s\mid r\in\mathbb{Z},s\in\mathbb{N}\})=T_1\cap T_2$.

\medskip
(4) Set $P=a^2H\cup abH\cup a^2bH\cup ab^2H$ and $Q=b^2H\cup abH\cup a^2bH\cup ab^2H$. Since $(ab)^2=a^2b^2$, $(ab)^3=(a^2b)(ab^2)$, $(a^2b)^2=a^2(ab)^2$ and $(ab^2)^2=b^2(ab)^2$, we infer that each non-empty prime $s$-ideal of $H$ contains $P$ or $Q$. It is straightforward to prove that $P$ and $Q$ are non-empty prime $s$-ideals of $H$. Now it is easy to see that $\mathfrak{X}(H)=\{P,Q\}$. Observe that $H_P=T_1$ and $H_Q=T_2$. Therefore, $H=H_P\cap H_Q$ is seminormal and weakly Krull by (3).
\end{proof}

By Theorem~\ref{theorem3}, we know that every seminormal weakly factorial monoid whose root closure is a Krull monoid has to be half-factorial. On the contrary, Example~\ref{example3} shows that a seminormal (reduced affine) weakly Krull monoid whose root closure is factorial need not even be a transfer Krull monoid.

\begin{example}\label{example4} Let $F$ be the free abelian monoid with basis $\{a,b,c\}$ and quotient group $K$ and let $H$ be the submonoid of $K$ generated by $\{a,ac,ab,abc,ab^2,ab^2c,a^2b^5,a^2b^5c,c^2,c^{-2}\}$.
\begin{enumerate}
\item[(1)] $H$ is an affine monoid with quotient group $K$ and $\mathcal{A}(H)=\{ac^t,abc^t,ab^2c,a^2b^5c^t\mid t\in\mathbb{Z}\}$.
\item[(2)] $H\subseteq\widetilde{H}$ is inert and, in particular, $H$ is transfer Krull.
\item[(3)] $H$ is seminormal.
\item[(4)] $H$ is neither Krull nor half-factorial.
\end{enumerate}
\end{example}

\begin{proof} (1) This is clear.

(2) We have that $\widetilde{H}=[\{a,ab,ab^2,a^2b^5,c,c^{-1}\}]$ and $H\setminus H^{\times}=\widetilde{H}\setminus\widetilde{H}^{\times}$. Let $x,y\in\widetilde{H}$ be such that $xy\in H$. We have to show that $x\varepsilon,y\varepsilon^{-1}\in H$ for some $\varepsilon\in\widetilde{H}^{\times}$. The statement clearly holds if $x,y\in H$. Now let $x\not\in H$ or $y\not\in H$. Without restriction let $x\not\in H$. Then $x\in\widetilde{H}^{\times}$. Set $\varepsilon=x^{-1}$. Then $\varepsilon\in\widetilde{H}^{\times}$, $x\varepsilon=1\in H$ and $y\varepsilon^{-1}=xy\in H$.

(3) This is clear, since $H\setminus H^{\times}=\widetilde{H}\setminus\widetilde{H}^{\times}$.

(4) Since $c\in\widetilde{H}\setminus H$, we have that $H$ is not Krull. Since $(ab)^5=(a^2b^5)a^3$, we obtain that $H$ is not half-factorial.
\end{proof}

\bigskip
\noindent {\bf ACKNOWLEDGEMENTS.} We want to thank A. Geroldinger for helpful discussions and for many suggestions and comments which improved the quality of this paper.

\providecommand{\bysame}{\leavevmode\hbox to3em{\hrulefill}\thinspace}
\providecommand{\MR}{\relax\ifhmode\unskip\space\fi MR}
\providecommand{\MRhref}[2]{\href{http://www.ams.org/mathscinet-getitem?mr=#1}{#2}}
\providecommand{\href}[2]{#2}

\end{document}